\pdfoutput=1
\documentclass{article}
\usepackage{pgfplots}
\usepackage{calc}
\usepackage{algorithm2e}
\usepackage{xcolor}
\usepackage{hyperref}
\usepackage{amsmath,amssymb,amsthm}
\usepackage{bm}

\hypersetup{colorlinks=true,citecolor=red}
\usetikzlibrary{calc,positioning,intersections,external}
\usepgfplotslibrary{groupplots,fillbetween}
\newtheorem{definition}{Definition}
\newtheorem{lemma}[definition]{Lemma}
\newtheorem{theorem}[definition]{Theorem}

\newtheorem{corr}[definition]{Corollary}

\newcommand{\R}{\mathbb R}
\newcommand*\phantomas[3][c]{%
\ifmmode
    \makebox[\widthof{$#2$}][#1]{$#3$}%
\else
    \makebox[\widthof{#2}][#1]{#3}%
\fi
}

\newcommand{\revise}[1]{{#1}}

\begin{document}

\title{\texttt{trlib}: A vector-free implementation of the GLTR method for iterative solution of the trust region problem}

\author{
  F. Lenders \and C. Kirches \and A. Potschka
}

    \maketitle

    \begin{abstract}

        We describe \texttt{trlib}, a library that implements a variant
        of Gould's Generalized Lanczos method (Gould et al. in
        SIAM J.~Opt.\ 9(2), 504--525, 1999) for solving the trust region problem.

        Our implementation has several distinct features that set it apart from
        preexisting ones. We implement both conjugate gradient (CG) and Lanczos
        iterations for 
        assembly of Krylov subspaces. A vector- and matrix-free
        reverse communication interface allows the use of most general data
        structures, such as those arising after discretization of function space
        problems.
        The hard case of the trust region
        problem frequently arises in sequential methods for nonlinear
        optimization.  In this implementation, we made an effort to fully
        address the hard case in an exact way by considering all invariant
        Krylov subspaces.

        We investigate the numerical performance of \texttt{trlib} on the full
        subset of unconstrained problems of the \texttt{CUTEst} benchmark set.
        In addition to this, interfacing the PDE discretization toolkit
        \texttt{FEniCS} with \texttt{trlib} using the vector-free reverse
        communication interface is demonstrated for a family of PDE-constrained
        control trust region problems adapted from the \texttt{OPTPDE} collection.


    \end{abstract}

    \section{Introduction}

In this article, we are concerned with solving the trust region problem, as it
frequently arises as a subproblem in sequential algorithms for nonlinear
optimization.

For this, let $ \mathcal H $ denote a Hilbert space with inner product $ \langle
\cdot, \cdot \rangle $ and norm $ \Vert \cdot \Vert $. Then, $ H: \mathcal H \to
\mathcal H $ denotes a self-adjoint, bounded operator on $\mathcal H$.
\revise{We assume that $H$ has compact negative part, which implies sequential weak lower semicontinuity of the mapping $ x \mapsto \langle x, Hx \rangle $, cf.~\cite{Hestenes1951} for details and a motivation}. \revise{In particular, we assume that self-adjoint, bounded operators $ P$ and $K $ exist  on $ \mathcal H $, such that $ H = P - K $, that $ K $ is compact, and that $ \langle x, Px \rangle \ge 0 $ for all $ x \in \mathcal H $}.
The operator $ M: \mathcal H \to \mathcal H $ is 
self-adjoint, bounded and coercive such that it induces an inner product  $ \langle
\cdot, \cdot \rangle_M $ with corresponding norm $ \Vert \cdot \Vert_M $ via $
\langle x, y \rangle_M := \langle x, M y \rangle $ and $ \Vert x \Vert_M :=
\sqrt{ \langle x, x \rangle_M } $.
Furthermore, let $ \mathcal X \subseteq \mathcal H $ be a closed subspace.

\label{def:tr}
The trust region subproblem we are interested in reads
\begin{align}
    \label{eq:tr}\tag{$\textup{TR}(H,g,M,\Delta,\mathcal X)$}
    \hspace*{-2cm}\left\{
        \quad\begin{array}{cl}\displaystyle
            \min_{x \in \mathcal H}\quad& \tfrac 12 \langle x, Hx \rangle + \langle x, g \rangle \\
            \textup{s.t.}\quad & \Vert x \Vert_M \le \Delta,\\
                               & x \in \mathcal X,
        \end{array}
    \right.
\end{align}
with $g\in\mathcal H$, objective function $ q(x) := \tfrac 12
\langle x, Hx \rangle + \langle x, g \rangle $, and trust region radius $ \Delta > 0 $.
Usually we take $ \mathcal X = \mathcal H $ but will also consider truncated versions where $ \mathcal X $ is a finite dimensional subspace of $ \mathcal H $.

Readers who are less comfortable with the function space setting may think of
$H$ as a symmetric positive definite matrix, of $\mathcal H$ as $\R^n$, and of
$M$ as the identity on $ \mathbb R^n $ inducing the standard scalar product and the euclidean norm
$\Vert \cdot\Vert_2$.
We follow the convention to indicate coordinate vectors $ \boldsymbol x \in \R^n $ with boldface letters.
\medskip

\subsubsection*{\bf Related Work}

Trust Region Subproblems are an important ingredient in modern optimization
algorithms as globalization mechanism. The monography \cite{Conn2000} provides
an exhaustive overview on Trust Region Methods for nonlinear programming, mainly
for problems formulated in finite-dimensional spaces. For trust region
algorithms in Hilbert spaces, we refer to
\cite{Kelley1987,Toint1988,Heinkenschloss1993,Ulbrich2000} and for Krylov
subspace methods in Hilbert space \cite{Guennel2014}.
In \cite{Absil2007}
applications of trust region subproblems formulated on Riemannian manifolds are
considered. Recently, trust region-like algorithms with guaranteed complexity
estimates in relation to the KKT tolerance have been proposed
\cite{Cartis2011,Cartis2011a,Curtis2016}. The necessary ingredients in the
subproblem solver for the algorithm investiged by Curtis and Samadi
\cite{Curtis2016} have been implemented in \texttt{trlib} as well.

Solution algorithms for trust region problems can be classified into direct
algorithms that make use of matrix factorizations and iterative methods that
access the operators $ H $ and $ M $ only via evaluations $ x \mapsto Hx $ and $
x \mapsto M x $ or $ x \mapsto M^{-1} x $. For the Hilbert space context, we
are interested in the latter class of algorithms. We refer to \cite{Conn2000}
and the references therein for a survey of direct algorithms, but point out the
algorithm of Mor\'{e} and Sorensen \cite{More1983} that will be used on a
specific tridiagonal subproblem, as well as the work of Gould et
al.~\cite{Gould2010}, who use higher order Taylor models to obtain high order
convergence results. The first iterative method was based on the conjugate
gradient process, and was proposed independently by Toint~\cite{Toint1981} and
Steihaug~\cite{Steihaug1983}. Gould et al.~\cite{Gould1999} proposed an
extension of the Steihaug-Toint algorithm. There, the Lanczos algorithm is used
to build up a nested sequence of Krylov spaces, and tri-diagonal trust region
subproblems are solved with a direct method. This idea also forms the basis for
our implementation. Hager \cite{Hager2001} considers an approach that builds on
solving the problem restricted to a sequence of subspaces that use SQP iterates
to accelerate and ensure quadratic convergence. Erway et
al.~\cite{Erway2009,Erway2010} investigate a method that also builds on a
sequence of subspaces built from accelerator directions satisfying optimality
conditions of a primal-dual interior point method.
\revise{In the methods of Steihaug-Toint and Gould, the operator $ M $ is used to
define the trust region norm and acts as preconditioner in the Krylov subspace algorithm.
The method of Erway et al.~allows to use a preconditioner that is independent
of the operator used for defining the trust region norm.
The trust region problem can equivalently be stated as generalized eigenvalue problem.
Approaches based on this characterization} are studied by Sorensen \cite{Sorensen1997}, Rendl and
Wolkowicz \cite{Rendl1997}, and Rojas et al.~\cite{Rojas2000,Rojas2008}.

\subsubsection*{\bf Contributions}

We introduce \texttt{trlib} which is a
new vector-free implementation of the GLTR \revise{(Generalized Lanczos Trust Region)} method for solving the trust region
subproblem. We assess the performance of this implementation on trust region
problems obtained from the set of unconstrained nonlinear minimization problems
of the CUTEst benchmark library, as well as on a number of examples formulated
in Hilbert space that arise from PDE-constrained optimal control.

\subsubsection*{\bf Structure of the Article} The remainder of this article is
structured as follows. In \S\ref{sec:ExUniq}, we briefly review conditions for
existence and uniqueness of minimizers. The GLTR methods for iteratively solving the 
trust region problem is presented in \S\ref{sec:GLTR} in detail. Our implementation,
\texttt{trlib} is introduced in \S\ref{sec:TRLIB}. Numerical results for
trust-region problems arising in nonlinear programming and in PDE-constrained
control are presented in \S\ref{sec:results}. Finally, we offer a summary and
conclusions in \S\ref{sec:conclusions}.

\section{Existence and Uniqueness of Minimizers}
\label{sec:ExUniq}

In this section, we briefly summarize the main results about existence and
uniqueness of solutions of the trust region subproblem. We first note that our
introductory setting implies the following fundamental properties:

\begin{lemma}[Properties of~\eqref{eq:tr}]~\\[-1em]
    \label{lem:fundamentals}
    \begin{enumerate}
        
        \item The mapping $ x \mapsto \langle x, H x \rangle $ is sequentially
        weakly lower semicontinuous, and Fr\'echet differentiable for every $ x
        \in \mathcal H $.
        
        \item The feasible set $\mathcal F := \{ x \in \mathcal H \, \vert \,
        \Vert x \Vert_M \le \Delta \} $ is bounded and weakly closed.
        
        \item The operator $ M $ is surjective.

    \end{enumerate}
    \begin{proof}
      $H \revise{{}= P - K}$ \revise{with compact $K$}, so (1) follows from \cite[Thm 8.2]{Hestenes1951}.
        Fr\'echet differentiability follows from boundedness of $H$. Boundedness of
        $\mathcal F$ follows from coercivity of $M$. Furthermore, $\mathcal F$ is
        obviously convex and strongly closed, hence weakly closed.     Finally, (3)
        follows by the Lax-Milgram theorem~\cite[ex. \revise{7}.19]{Clarke2013}\revise{:
          By boundedness of $ M $, there is $ C > 0 $ with $ \vert \langle x, M y \rangle \vert \le C \Vert x \Vert \, \Vert y \Vert $. The coercitivity assumption implies existence of $ c > 0 $ such that $ \langle x, Mx \rangle \ge c \Vert x \Vert^2 $ for $ x, y \in \mathcal H $.
         Then, $ M $ satisfies the assumptions of the Lax-Milgram theorem. Given $ z \in \mathcal H $, application of this theorem yields $ \xi \in \mathcal H $ with $ \langle x, M \xi \rangle = \langle x, z \rangle $ for all $ x \in \mathcal H $.
        Thus $ M \xi = z $.
        }
    \end{proof}
\end{lemma}

\begin{lemma}[Existence of a solution]~\\[.5em]
    Problem~\eqref{eq:tr} has a solution.
\end{lemma}

\begin{proof}
    By Lemma~\ref{lem:fundamentals}, the objecive functional $q$ is sequentially
    weakly lower semicontinuous and the feasible set  $\mathcal F$ is weakly
    closed and bounded, the claim follows then from a generalized Weierstrass Theorem \cite[Ch. 7]{Kurdila2005}.
\end{proof}

To present optimality conditions for the trust region subproblem, we first
present a helpful lemma on the change of the objective function between two
points on the trust region boundary.

\begin{lemma}[Objective Change on Trust Region Boundary]\label{lem:obj-change}~\\[.5em]
    Let $ x^0, x^1 \in \mathcal H $ with $ \Vert x^i \Vert_M = \Delta $ for $ i = 0, 1 $ be boundary points
    of~\eqref{eq:tr}, and let $ \lambda \ge 0 $ satisfy $ (H+\lambda M) x^0 + g = 0 $.
    Then $ d = x^1 - x^0 $ satisfies $ q(x^1) - q(x^0) = \tfrac 12 \langle d, (H+\lambda M) d \rangle $.
\end{lemma}
\begin{proof}
    Using $ 0 = \Vert x^1 \Vert^2_M - \Vert x^0 \Vert^2_M = \langle x^0 + d, M (x^0 + d) \rangle - \langle x^0, M x^0 \rangle = \langle d, M d \rangle + 2 \langle x^0, M d \rangle $ and $ g = - (H+\lambda M) x^0 $ we find
    \begin{align*}
      q(x^1) - q(x^0) & = \tfrac 12 \langle d, Hd \rangle + \langle d, Hx^0 \rangle + \langle g, d \rangle = \tfrac 12 \langle d, Hd \rangle - \overbrace{ \lambda \langle x^0, Md \rangle}^{-\tfrac 12 \lambda \langle d, Md \rangle} \\
        & = \tfrac 12 \langle d, (H + \lambda M) d \rangle. \qedhere
    \end{align*}
\end{proof}

Necessary optimality conditions for the finite dimensional problem, see e.g.
\cite{Conn2000}, generalize in a natural way to the Hilbert space context.

\begin{theorem}[Necessary Optimality Conditions] \label{thm:noc}~\\[.5em]
    Let $ x^* \in \mathcal H $ be a global solution of~$(\textup{TR}(H,g,M,\Delta,\mathcal H))$. Then there is $
    \lambda^* \ge 0 $ such that
    \begin{enumerate}
        \renewcommand{\theenumi}{(\alph{enumi})}
        \item $ (H + \lambda^* M) x^* + g = 0 $,
        \item $ \Vert x^* \Vert_M - \Delta \le 0 $,
        \item $ \lambda^* ( \Vert x^* \Vert_M - \Delta ) = 0 $,
        \item $ \langle d, (H + \lambda^* M) d \rangle \ge 0 $ for all $ d \in \mathcal H $.
    \end{enumerate}
\end{theorem}

\begin{proof}
    Let $ \sigma: \mathcal H \to \mathbb R, \sigma(x) := \langle x, M x \rangle
    - \Delta^2 $, so that the trust region constraint becomes $ \sigma(x) \le 0 $.
    The function $ \sigma $ is Fr\'echet-differentiable for all $ x \in \mathcal
    H $ with surjective differential provided $ x \neq 0 $ and satisfies constraint qualifications in that case.
    We may assume $ x^* \neq 0 $ as the theorem holds for $ x^* = 0 $ (then $g=0$) for elementary
    reasons.

    Now if $ x^* $ is a \revise{global} solution of~$(\textup{TR}(H,g,M,\Delta,\mathcal H))$, conditions (a)--(c) are necessary
    optimality conditions, cf.~\revise{\cite[Thm 9.1]{Clarke2013}}.
    
    \revise{To prove (d), we distinguish three cases:
        \begin{itemize}
            \item $ \Vert x \Vert_M = \Delta $ and $ d \in \mathcal H $ with $ \langle d, Mx^* \rangle \neq 0 $:
                Given such $ d $, there
                is $ \alpha \in \mathbb R \setminus \{0\} $ with $ \Vert x^* + \alpha d \Vert_M = \Delta
                $. Using Lemma~\ref{lem:obj-change} yields $ \langle d, (H+\lambda^* M) d \rangle
                = \tfrac 2{\alpha^2} ( q(x^* + \alpha d) - q(x^*) ) \ge 0 $ since $ x^* $ is a global solution.
            \item  $ \Vert x \Vert_M = \Delta $ and $ d \in \mathcal H $ with $ \langle d, Mx^* \rangle = 0 $:
                Since $ x^* \neq 0 $ and $ M $ is surjective, there is $ p \in H $ with $ \langle p, M x^* \rangle \neq 0 $,
                let $ d(\tau) := d + \tau p $ for $ \tau \in \mathbb R $. Then $ \langle d(\tau), M x^* \rangle \neq 0 $ for $ \tau \neq 0 $,
                by the previous case \begin{align*} 0 & \le \langle d(\tau), (H + \lambda^* M) d(\tau) \rangle \\ & = \langle d, (H + \lambda^* M) d \rangle + \tau \langle p, (H+\lambda^* M) d \rangle + \tau^2 \langle p, (H+\lambda^* M) p \rangle. \end{align*}
                Passing to the limit $ \tau \to 0 $ shows $ \langle d, (H+\lambda^* M) d \rangle \ge 0 $.
            \item $ \Vert x \Vert_M < \Delta $: Then $ \lambda^* = 0 $ by (c). Let $ d \in \mathcal H $ and consider $ x(\tau) = x^* + \tau d $, which is feasible for sufficiently small $ \tau $.
                By optimality and stationarity (a):
                \begin{align*}
                    0 \le q(x(\tau)) - q(x^*) = \tau \langle x^*, Hd \rangle + \tfrac{\tau^2}{2} \langle d, Hd \rangle + \tau \langle g, d \rangle = \tfrac{\tau^2}{2} \langle d, H d \rangle,
                \end{align*}
                thus $ \langle d, H d \rangle \ge 0 $. \qedhere
        \end{itemize}
    }
\end{proof}

\begin{corr}[Sufficient Optimality Condition]~\\[.5em]
    Let $ x^* \in \mathcal H $ and $ \lambda^* \ge 0 $
    such that (a)--(c) of Thm.~\ref{thm:noc} hold and  $ \langle d, (H + \lambda^*
    M) d \rangle > 0 $ holds for all $ d \in \mathcal H $. Then $ x^* $ is the
    unique global solution of~$(\textup{TR}(H,g,M,\Delta,\mathcal H))$.
\end{corr}

\begin{proof}
    This is an immediate consequence of Lemma~\ref{lem:obj-change}.
\end{proof}

    \section{The GLTR Method}
\label{sec:GLTR}

The GLTR (Generalized Lanczos Trust Region) method is an iterative method to
approximatively solve $(\textup{TR}(H,g,M,\Delta,\mathcal H))$ and has first been described in Gould et al.
\cite{Gould1999}. Our presentation follows the presentation there and only
deviates in minor details.

In every iteration of the GLTR process, problem~$(\textup{TR}(H,g,M,\Delta,\mathcal H))$ is restricted
to the Krylov subspace $ \mathcal K_i := \text{span}\{ (M^{-1}H)^j M^{-1}g \, \vert \, 0 \le j \le i \} $,
\begin{align}
    \label{eq:tr_krylovi}\tag{$\textup{TR}(H,g,M,\Delta,\mathcal K_i)$}
	\hspace*{-3cm}\left\{
		\quad\begin{array}{cl}\displaystyle
      \min_{x \in \mathcal H}\quad& \tfrac 12 \langle x, Hx \rangle + \langle x, g \rangle \\
      \textup{s.t.}\quad & \Vert x \Vert_M \le \Delta,\\
                         & x \in \mathcal K_i.
		\end{array}
	\right.
\end{align}

The following Lemma relates solutions of~\eqref{eq:tr_krylovi} to those of~$(\textup{TR}(H,g,M,\Delta,\mathcal H))$.

\begin{lemma}[Solution of the Krylov subspace trust region problem]~\\[.5em]
Let $ x^i $ be a global minimizer of~\eqref{eq:tr_krylovi} and $ \lambda^i $ the corresponding Lagrange multiplier. 
Then $ (x^i, \lambda^i) $ satisfies the global optimality conditions of $(\textup{TR}(H,g,M,\Delta,\mathcal H))$ (Thm.~\ref{thm:noc}) in the following sense:
\begin{enumerate}
    \renewcommand{\theenumi}{(\alph{enumi})}
    \item $ (H + \lambda^i M) x^i + g \, \perp_M \, \mathcal K_i $,
    \item $ \Vert x^i \Vert_M - \Delta \le 0 $,
    \item $ \lambda^i ( \Vert x^i \Vert_M - \Delta ) = 0 $,
    \item $ \langle d, (H + \lambda^i M) d \rangle \ge 0 $ for all $ d \in \mathcal K_i $.
\end{enumerate}
\end{lemma}
\begin{proof}
(b)--(d) are immediately obtained from Thm.~\ref{thm:noc} as $ \mathcal K_i \subseteq \mathcal H $ is a Hilbert space.
Assertion (a) follows from 
$x^\ast=x^i+x^\perp$ with $x^i\in\mathcal K_i$, $x^\perp\perp\mathcal K_i$ and
Thm.~\ref{thm:noc} for $x^i$.
\end{proof}
Solving problem~$(\textup{TR}(H,g,M,\Delta,\mathcal H))$ may thus be achieved by iterating the following
Krylov subspace process. Each iteration requires the solution of an instance of the
truncated trust region subproblem~\eqref{eq:tr_krylovi}.
\begin{algorithm}
    \DontPrintSemicolon
    \SetKwInOut{Input}{input}
    \SetKwInOut{Output}{output}
    \SetKw{KwReturn}{return}
    \Input{$H$, $M$, $g$, $\Delta$, \emph{\texttt{tol}}}
    \Output{$i$, $x^i$, $\lambda^i$}
    \BlankLine  
    \For{$i\geq 0$}{
      {Construct a basis for the $i$-th Krylov subspace $ \mathcal K_i $}\;
      {Compute a representation of $q(x)$ restricted to $ \mathcal K_i $}\;
      {Solve the subproblem~\eqref{eq:tr_krylovi} to obtain $(x^i,\lambda^i)$}\;    
      \lIf{$\Vert (H + \lambda^i M) x^i + g \Vert_{M^{-1}}\leq\texttt{tol}$}{\KwReturn}
    }
    \caption{Krylov subspace process for solving~$(\textup{TR}(H,g,M,\Delta,\mathcal H))$.}
    \label{alg:KrylovTR}
\end{algorithm}

Algorithm~\ref{alg:KrylovTR} stops the subspace iteration as soon as $ \Vert (H
+ \lambda^i M) x^i + g \Vert_{M^{-1}} $ is small enough.
The norm $ \Vert \cdot \Vert_{M^{-1}} $ is used in the termination criterion since it is the norm belonging to the dual of $ (\mathcal H, \Vert \cdot \Vert_M) $ and the Lagrange \revise{derivative representation} $ (H + \lambda^i M) x^i + g $ should be regarded as element of the dual.

\subsection{Krylov Subspace Buildup}

In this section, we present the preconditioned conjugate gradient (pCG) process
and the preconditioned Lanczos process (pL) for construction of Krylov subspace
bases. We discuss the transition from pCG to pL upon breakdown of the pCG process.

\subsubsection{Preconditioned Conjugate Gradient Process}

An $ H $-conjugate basis $(\hat p_j)_{0 \le j \le i} $ of $ \mathcal K_i $ may
be obtained using preconditioned conjugate gradient (pCG) iterations,
Algorithm~\ref{alg:pCG}.

\begin{algorithm}
    \DontPrintSemicolon
    \SetKwInOut{Input}{input}
    \SetKwInOut{Output}{output}
    \Input{$H$, $M$, $g^0$, $i\in\mathbb N$}
    \Output{$v^i$, $g^i$, $p^i$, $\alpha^{i-1}$, $\beta^{i-1}$}
    \BlankLine  
    {Initialize $\hat v^0 \gets M^{-1}g^0$, $\hat p^0 \gets -\hat v^0$}\;
    \For{$j\gets 0$ \KwTo $i-1$}{
      {$\alpha^j \gets {\langle \hat g^j, \hat v^j \rangle}/{\langle \hat p^j, H \hat p^j \rangle} = {\Vert \hat v^j \Vert_M}/{\langle \hat p^j, H \hat p^j \rangle}$}\;
      {$\hat g^{j+1}\gets \hat g^j + \alpha^j H \hat p_j$}\;
      {$\hat v^{j+1} \gets M^{-1} \hat g^{j+1}$}\;
      {$\beta^j \gets {\langle \hat g^{j+1}, \hat v^{j+1} \rangle}/{\langle \hat g^j, \hat v^j \rangle} = {\Vert \hat v^{j+1} \Vert_M^2}/{\Vert \hat v^j \Vert_M^2}$}\; 
      {$\hat p^{j+1} \gets - \hat v^{j+1} + \beta^j \hat p^j$}\;
    }
    \caption{Preconditioned conjugate gradient (pCG) process.}
    \label{alg:pCG}
\end{algorithm}

The stationary point $ s^i $ of $ q(x) $ restricted to the Krylov subspace $ \mathcal K_i $ is given by $ s^i = \sum_{j=0}^i \alpha^j \hat p^j $ and can thus be  computed using the recurrence 
\begin{align*}
  s^0 \gets \alpha^0 \hat p^0,\quad s^{j+1} \gets s^j + \alpha^{j+1} \hat p^{j+1},\ 0\leq j \leq N-1
\end{align*}
as an extension of Algorithm~\ref{alg:pCG}.
The iterates' $M$-norms $ \Vert s^i \Vert_M $ are monotonically increasing \cite[Thm 2.1]{Steihaug1983}.
Hence, as long as $ H $ is coercive on the subspace $ \mathcal K_i $ (this implies $ \alpha_j > 0 $ for $ 0 \le j \le i $) and $ \Vert s^i \Vert_M \le \Delta $, the solution to \eqref{eq:tr_krylovi} is directly given by $ s^i $.
Breakdown of the pCG process occurs if $ \alpha^i = 0 $.
In computational practice, if the criterion $|\alpha^i|\leq \varepsilon$ is violated, where $ \varepsilon \ge 0 $ is a suitable small tolerance,
it is possible -- and necessary -- to continue with Lanczos iterations, described next.

\subsubsection{Preconditioned Lanczos Process}

An $M$-orthogonal basis $ (p_j)_{0 \le j \le i} $ of $ \mathcal K_i $ may be obtained using the preconditioned Lanczos (pL) process, Algorithm~\ref{alg:PLCZ}, and permits to continue subspace iterations even after pCG breakdown.

\begin{algorithm}
    \DontPrintSemicolon
    \SetKwInOut{Input}{input}
    \SetKwInOut{Output}{output}
    \Input{$H$, $M$, $g^0$, $j\in\mathbb N$}
    \Output{$v^i$, $g^i$, $p^{i-1}$, $\gamma^{i-1}$, $\delta^{i-1}$}
    \BlankLine  
    {Initialize $ g^{-1} \gets 0$, $\gamma^{-1} \gets 1$, $v^0 \gets M^{-1}g^0$, $p^0 \gets v^0 $}\;
    \For{$i\gets 0$ \KwTo $j-1$}{
      {$\gamma^j \gets \sqrt{ \langle g^j, v^j \rangle } = \Vert g^j \Vert_{M^{-1}} = \Vert v^j \Vert_M$}\;
      {$p^j \gets  (1/{\gamma^j}) v^j =  (1/{\Vert v^j \Vert_M}) v^j$}\;
      {$\delta^j\gets \langle p^j, H p^j \rangle$}\;
      {$g^{j+1} \gets Hp^j - ({\delta^j}/{\gamma^j}) g^j - ({\gamma^j}/{\gamma^{j-1}}) g^{j-1}$}\;
      {$v^{j+1} \gets M^{-1} g^{j+1}$}
    }
    \caption{Preconditioned Lanczos (pL) process.}
    \label{alg:PLCZ}
\end{algorithm}

The following simple relationship holds between the Lanczos iteration data and the pCG iteration data, and may be used to initialize the
pL process from the final pCG iterate before breakdown:
\begin{alignat*}{3}
    \gamma^i & = \begin{cases} \Vert \hat v^0 \Vert_M, & i = 0 \\ {\sqrt{\beta^{i-1}}}/{\vert \alpha^{i-1} \vert}, & i \ge 1 \end{cases}, \qquad&
    \delta^i & = \begin{cases}  1/{\alpha^{0}}, & i = 0 \\  1/{\alpha^{i}} +  {\beta^{i-1}}/{\alpha^{i}}, & i \ge 1 \end{cases}, \\
    p^i & = 1/{\Vert \hat v_i \Vert_M}{\left[ \prod_{j=0}^{i-1} (- \text{sign} \, \alpha^{j})\right ]} \, \hat v_i, \qquad&
    g^i & = \gamma^j/{\Vert \hat v_i \Vert_M} \left[ \prod_{j=0}^{i-1}(- \text{sign} \, \alpha^{j}) \right] \, \hat g_i.
\end{alignat*}
In turn, breakdown of the pL process occurs if an invariant subspace of $H$ is
exhausted, in which case $ \gamma^i = 0 $. If this subspace does not span
$ \mathcal H $, the pL process may be restarted if provided with a vector $g^0$ that is $ M
$-orthogonal to the exhausted subspace.

The pL process may also be expressed in compact matrix form as 
\begin{align*}
    H P_i - M P_i T_i = g^{i+1} \boldsymbol{e_{i+1}}^T,\ \langle P_i, M P_i \rangle = I ,
\end{align*} 
with $ P_i $ being the matrix composed from columns $ p_0, \ldots, p_i $, and $
T_i $ the symmetric tridiagonal matrix with diagonal elements $ \delta^0,
\ldots, \delta^i $ and off-diagonal elements $ \gamma^1, \ldots, \gamma^i $.

As $ P_i $ is a basis for $ \mathcal K_i $, every $ x \in \mathcal K_i $ can be written as $ x =
P_i \boldsymbol h $ with a coordinate vector $ \boldsymbol h \in \mathbb R^{i+1} $. Using the compact form
of the Lanczos iteration, one can immediately express the quadratic form in this
basis as $ q(x) = \tfrac 12 \langle \boldsymbol h, T_i \boldsymbol h \rangle + \gamma^0 \langle \boldsymbol{e_1}, \boldsymbol h \rangle
$. Similarly, $ \Vert x \Vert_M = \Vert \boldsymbol h \Vert_2 $. Solving
\eqref{eq:tr_krylovi} thus reduces to solving $\textup{TR}(T_i,\gamma^0 \boldsymbol{e_1},I,\Delta,\mathbb R^{i+1})$ on
$ \mathbb R^{i+1} $ and recovering $ x = P_i \boldsymbol h $.

\subsection{Easy and Hard case of the Tridiagonal Subproblem}

As just described, using the tridiagonal representation $T_i$ of $H$ on the
basis $P_i$ of the $i$-th iteration of the pL process, the trust-region
subproblem $ \text{TR}(T_i, \gamma^0 \boldsymbol{e_1}, I, \Delta, \mathbb R^{i+1}) $ needs to be solved. For
notational convenience, we drop the iteration index $ i $ from $ T_i $ in the
following. Considering the necessary optimality conditions of
Thm.~\ref{thm:noc},  it is natural to define the mapping
\begin{align*}
    \lambda \mapsto \boldsymbol{x}(\lambda) := (T+\lambda I)^{+} (-\gamma^0 \boldsymbol{e_1}) \textup{ for }
  \lambda \in I := [\max\{0, -\theta_{\min} \}, \infty), 
\end{align*} 
where $ \theta_{\min} $ denotes the smallest eigenvalue of $ T $, and the superscript $+$ denotes the Moore-Penrose pseudo-inverse.
On $ I $, $ T + \lambda I $ is positive semidefinite. The following definition
relates $\boldsymbol{x}(\lambda^\ast)$ to a global minimizer $ (\boldsymbol{x^*}, \lambda^*) $ of $\textup{TR}(T_i,\gamma^0 \boldsymbol{e_1},I,\Delta, \mathbb R^{i+1})$.

\begin{definition}[Easy Case and Hard Case]~\\[.5em] Let $ (\boldsymbol{x^*}, \lambda^*) $
satisfy the necessary optimality conditions of Thm.~\ref{thm:noc}.\\ If $
\langle \gamma^0 \boldsymbol{e_1}, \textup{Eig}(\theta_{\min}) \rangle \neq 0 $, we say that
$T$ satisfies the \emph{easy case}. Then,  $ \boldsymbol{x^*} = \boldsymbol{x}(\lambda^*) $.\\ If $
\langle \gamma^0 \boldsymbol{e_1}, \textup{Eig}(\theta_{\min}) \rangle = 0 $, we say that $T$
satisfies the \emph{hard case}. Then, $ \boldsymbol{x^*} = \boldsymbol{x}(\lambda^*) + \boldsymbol{v} $ with suitable $
\boldsymbol v \in \textup{Eig}(\theta_{\min}) $.
Here $ \textup{Eig}(\theta) = \{ \boldsymbol v \in \mathbb R^{i+1} \, \vert \, T \boldsymbol v = \theta \boldsymbol v \} $ denotes the eigenspace of $ T $ associated to $ \theta $.
\end{definition}

With the following theorem, Gould et al.~in \cite{Gould1999} use the the
irreducible components of $ T $ to give a full description of the solution
$x(\lambda^*) + v$ in the hard case.
\begin{theorem}[Global Minimizer in the Hard Case]~\\[.5em]
\label{thm:hard-case-min}
    Let $ T = \textup{diag}(R_1, \ldots, R_k) $ with irreducible tridiagonal matrices $ R_j $ and let $ 1 \le \ell \le k $ be the smallest index for which $ \theta_{\min}(R_\ell) = \theta_{\min}(T) $ holds.
    Further, let  $ \boldsymbol{x_1}(\theta) = (R_1 + \theta I)^+(- \gamma^0 \boldsymbol{e_1}) $ and let $ (\boldsymbol{x_1^*}, \lambda_1^*) $ be a KKT-tuple corresponding to a global minimum of $ \textup{TR}(R_1, \gamma^0 \boldsymbol{e_1}, I, \Delta, \mathbb R^{r_1}) $, $ \boldsymbol{x_1^*} = \boldsymbol{x_1}(\lambda_1^*) $.\\[.5em] 
    If $ \lambda_1^* \ge -\theta_{\min} $, then $ x^* = (\boldsymbol{x_1}(\lambda_1^*)^T,\ \boldsymbol{0},\ \ldots,\ \boldsymbol{0})^T$ satisfies Thm.~\ref{thm:noc} for $ \text{TR}(T, \gamma^0 \boldsymbol{e_1}, I, \Delta, \mathbb R^{i+1}) $.\\[.5em]
    If $ \lambda_1^* <   -\theta_{\min} $, then $ x^* = (\boldsymbol{x_1}(-\theta_{\min})^T,\ \boldsymbol 0,\ \ldots,\ \boldsymbol 0,\ \boldsymbol v^T,\boldsymbol 0,\ \ldots,\ \boldsymbol 0)^T$, with $v \in \textup{Eig}(R_\ell, \theta_{\min}) $ such that $ \Vert \boldsymbol{x^*} \Vert^2_2 = \Vert \boldsymbol{x_1}(-\theta_{\min}) \Vert^2_2 + \Vert \boldsymbol v \Vert^2_2 = \Delta^2 $ satisfies Thm.~\ref{thm:noc} for $ \text{TR}(T, \gamma^0 \boldsymbol{e_1}, I, \Delta, \mathbb R^{i+1}) $.
\end{theorem}
In particular, as long as $ T $ is irreducible, the hard case does not occur.
\revise{A symmetric tridiagonal matrix $ T $ is irreducible, if and only if all it's offdiagonal elements are non-zero.}
For the tridiagonal matrices arising from Krylov subspace iterations, this is
the case as long as the pL process does not break down.

\subsection{Solving the Tridiagonal Subproblem in the Easy Case}
\label{sec:easy-tri}

Assume that $ T $ is irreducible, and thus satisfies the easy case.. Solving the
tridiagonal subproblem amounts to checking whether the problem admits an
interior solution and, if not, to finding a value $ \lambda^* \ge \max\{0,
-\theta_{\min}\} $ with $ \Vert x(\lambda^*) \Vert = \Delta $.

We follow Mor\'e and Sorensen~\cite{More1983}, who define $ \sigma_p(\lambda) :=
\Vert \boldsymbol x(\lambda) \Vert^p - \Delta^p $ and propose the Newton iteration 
\begin{align*}
\lambda^{i+1}\gets \lambda^i-{\sigma_p(\lambda^i)}/{\sigma_p'(\lambda^i)}=
\lambda^i-\frac{\Vert \boldsymbol x(\lambda^i) \Vert^p - \Delta^p }{{p \Vert \boldsymbol x(\lambda^i) \Vert^{p-2} \langle \boldsymbol x(\lambda^i),
\boldsymbol x'(\lambda^i) \rangle}},\ i \geq 0,
\end{align*}
with $\boldsymbol x'(\lambda) = - (T+\lambda I)^{+} \boldsymbol x(\lambda) $, to find a root of $
\sigma_{-1}(\lambda) $.  Provided that the initial value $\lambda^0$ lies in the
interval $ [\max\{0, - \theta_{\min}\}, \lambda^*] $, such that $ (T + \lambda^0
I) $ is positive semidefinite, $ \Vert \boldsymbol x(\lambda^0) \Vert \ge \Delta $, and
 no safeguarding of the Newton iteration is necessary, it can be shown that
this leads to a sequence of iterates in the same interval that converges to $
\lambda^*$ at globally linear and locally quadratic rate,
cf.~\cite{Gould1999}.

Note that $ \lambda^* > - \theta_{\min} $ as $ \sigma_{-1}(\lambda) $ has a singularity in $ - \theta_{\min} $ but $ \sigma_{-1}(\lambda^*) = 1/\Delta $ and it thus suffices to consider $ \lambda > \max\{0, - \theta_{\min}\} $.

Both the function value and derivative require the solution of a linear system
of the form $ (T + \lambda I) \boldsymbol w = \boldsymbol b $. As $ T + \lambda I $ is tridiagonal,
symmetric positive definite, and of reasonably small dimension, it is
computationally feasible to use a tridiagonal Cholesky decomposition for this.

Gould et al.~in \cite{Gould2010} improve upon the convergence result by
considering higher order Taylor expansions of $ \sigma_p(\lambda) $ and values
$ p\neq-1 $ to obtain a method with locally quartic convergence.

\subsection{The Newton initializer}
\label{sec:newton-ini}

Cheap oracles for a suitable initial value $ \lambda^0 $ may be available,
including, for example, zero or the value $ \lambda^* $ of the previous
iteration of the pL process. If these fail, it becomes necessary to compute $
\theta_{\min} $. To this end, we follow Gould et al.~\cite{Gould1999} and
Parlett and Reid~\cite{Parlett1981}, who define the Parlett-Reid Last-Pivot
function $ d(\theta) $:
\begin{definition}[Parlett-Reid Last-Pivot Function]
\label{def:prlpf}
\begin{align*}
  d(\theta) := \begin{cases} d_{i},& \parbox{.7\textwidth}{$\text{if there exists } (d_0,\ldots,d_i) \in (0,\infty)^i \times \mathbb R$, and $ L $ unit lower triangular {such that}  $T - \theta I = L\, \textup{diag}(d_0, \ldots, d_i) \, L^T$} \\
        -\infty,& {otherwise.} \end{cases}
\end{align*}
\end{definition}
Since $ T $ is irreducible, its eigenvalues are simple \cite[Thm 8.5.1]{Golub1996} and $ \theta_{\min} $ is
given by the unique value $ \theta \in \mathbb R $ with $ T - \theta I $
singular and positive semidefinite, or, equivalently, $ d(\theta) = 0 $.

A safeguarded root-finding method is used to determine $ \theta_{\min} $ by
finding the root of $ d(\theta) $. An interval of safety $ [\theta^k_\ell,
\theta^k_\textup{u}]$ is used in each iteration and a guess $ \theta^k \in
[\theta^k_\ell, \theta^k_\textup{u}] $ is chosen. Gershgorin bounds may be used to
provide an initial interval \cite[Thm 7.2.1]{Golub1996}. Depending on the sign of $
d(\theta) $ the interval of safety is then contracted to $ [\theta^k_\ell,
\theta^k] $ if $ d(\theta^k) < 0 $ and to $ [\theta^k, \theta^k_\textup{u}] $ if
$ d(\theta^k) \ge 0 $ as the interval of safety for the next iteration. One choice
for $ \theta^k $ is bisection. Newton steps as previously described may be taken
advantage of if they remain inside the interval of safety.

\begin{figure}[ht]
  \centering
    \includegraphics{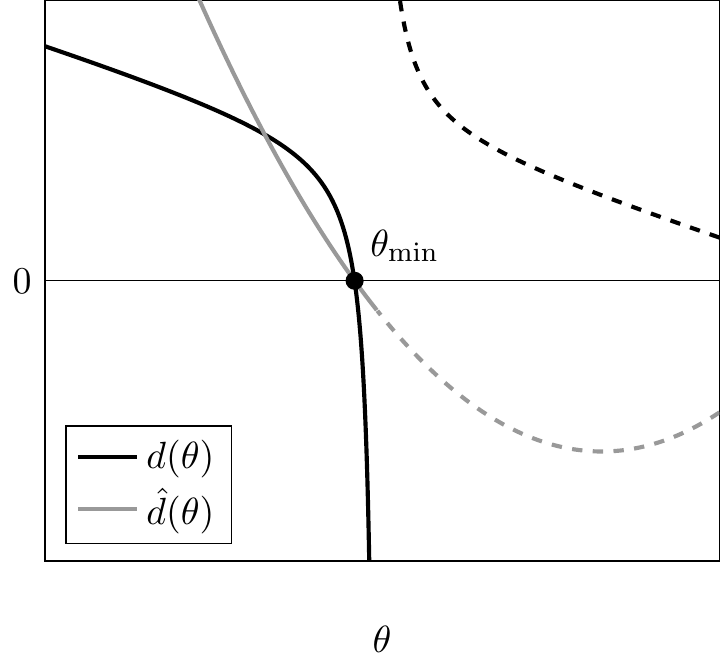}
    \caption{The Parlett-Reid last-pivot function $ d(\theta) $ and the lifted function $ \hat d(\theta) $ have the common zero $ \theta_{\min} $.
  Dashed lines show the analytic continuation of the right hand side of $ d(\theta) = \prod_j (\theta - \theta_j) / \prod_j (\theta - \hat \theta_j) $ into the region where $ d(\theta) = - \infty $.}
\end{figure}

For sucessive pL iterations, the fact that the tridiagonal matrices grow by one
column and row in each iteration may be exploited to save most of the
computational effort involved.  As noted by
Parlett and Reid \cite{Parlett1981}, the reccurence to compute the $d_i$ \revise{via Cholesky decomposition of $ T - \theta I $}
in~Def.~\ref{def:prlpf} is \revise{identical with the recurrence} that results from applying a Laplace expansion for \revise{the determinant of} tridiagonal matrices \cite[\S 2.1.4]{Golub1996}.
\revise{Comparing the recurrences thus} yields the explicit formula
\begin{align}
  \label{eq:d-recurrence}
    d(\theta) = \frac{\det(T - \theta I)}{\det(\hat T - \theta I)} 
    = \revise{-} \frac{\prod_{j} (\theta - \theta_j)}{\prod_{j} (\theta - \hat \theta_j)},
\end{align}
where $ \hat T $ denotes the principal submatrix of $ T $ obtained by erasing
the last column and row, and $\theta_j$ and $\hat\theta_j$ enumerate the
eigenvalues of $T$ and $\hat T$, respectively. \revise{The right hand side is obtained by identifying numerator and denominator with the characteristic polynomials of $ T $ and $ \hat T $, and by factorizing these.}

It becomes apparent that $ d(\theta) $ has a pole of first order in $ \hat
\theta_{\min} $. After lifting this pole, the function $ \hat d(\theta) :=
(\theta - \hat \theta_{\min}) d(\theta) $ is smooth on a larger interval. When
iteratively constructing the tridiagonal matrices in successive pL iterations,
the value $\hat \theta_{\min}$ is readily available and it becomes preferrable to use $
\hat d(\theta) $ instead of $d(\theta)$ for root finding.

\subsection{Solving the Tridiagonal Subproblem in the Hard Case}

If the hard case is present, the decomposition of $ T $ into irreducible components has to be determined.
This is given in a natural way by Lanczos breakdown.
Every time the Lanczos process breaks down and is restarted with a vector $ M $-orthogonal to the previously considered Krylov subspaces,
a new tridiagonal block is obtained.
Solving the problem in the hard case then amounts to applying Theorem~\ref{thm:hard-case-min}:
First all smallest eigenvalue $ \theta_i $ of the irreducible blocks $ R_i $ have to be determined as well as the KKT tuple $ (\boldsymbol{x_1^*}, \lambda_1^*) $ by solving the easy case for $ \text{TR}(R_1, \gamma^0 \boldsymbol{e_1}, I, \Delta, \mathbb R^{r_1}) $.
Again, let $ \ell $ be the smallest index $ i $ with minimial $ \theta_i $.
In the case $ \lambda_1^* \ge -\theta_{\ell} $, the global solution is given by $ x^* = ((\boldsymbol{x_1^*})^T,\ \boldsymbol 0,\ \ldots, \boldsymbol 0)^T $.
On the other hand if $ \lambda_1^* < - \theta_\ell $ the eigenspace of $ R_\ell $ corresponding to $ \theta_{\ell} $ has to be obtained.
As $ R_\ell $ is irreducible, all eigenvalues of $ R_\ell $ are simple and an eigenvector $ \boldsymbol{\tilde v} $ spanning the desired eigenspace can be obtained for example by inverse iteration \cite[\S 8.2.2]{Golub1996}.
The solution is now given by $ \boldsymbol{x^*} = (\boldsymbol{x_1(-\theta_{\ell})}^T,\ \boldsymbol 0,\ \boldsymbol v^T,\ \boldsymbol 0)^T $ with $ \boldsymbol{x_1}(-\theta_{\min}) = (R_1 - \theta_{\ell} I)^{-1}(-\gamma^0 \boldsymbol{e_1}) $ and $ \boldsymbol v := \alpha \boldsymbol{\tilde v} $ where $ \alpha $ has been chosen as the root of the scalar quadratic equation $ \Delta^2 = \Vert \boldsymbol{x_1}(-\theta_{\min}) \Vert^2 + \alpha^2 \Vert \boldsymbol{\tilde v} \Vert^2 $ that leads to the smaller objective value.

    \section{Implementation \texttt{trlib}}
\label{sec:TRLIB}
In this section, we present details of our implementation \texttt{trlib} of the
GLTR method. 

\subsection{Existing Implementation} 

The GLTR reference implementation is the
software package \texttt{GLTR} in the optimization library \texttt{GALAHAD}
\cite{GALAHAD}. This Fortran 90 implementation uses conjugate gradient
iterations exclusively to build up the Krylov subspace, and provides a reverse
communication interface that requires exchange vector data to be stored as 
contiguous arrays in memory.

\subsection{\texttt{trlib} Implementation} 

Our implementation is called \texttt{trlib}, short for \emph{trust region
library}.  It is written in plain ANSI C99 code, and has been made available as
open source \cite{Lenders2016}. We provide a reverse communication interface in
which only scalar data and requests for vector operations are exchanged,
allowing for great flexibility in applications.

Beside the stable and efficient conjugate gradient iteration we also implemented the Lanczos iteration and a crossover mechanism to expand the Krylov subspace, as we frequently found applications in the context of constrained optimization with an SLEQP algorithm \cite{Byrd2003,Lenders2015} where conjugate gradient iterations broke down whenever directions of tiny curvature have been encountered.

\subsection{Vector Free Reverse Communication Interface}

The implementation is built around a reverse communication calling paradigm.
To solve a trust region subproblem, the according library function has to be repeatedly called by the user and after each call the user has to perform a specific action indicated by the value of an output variable.
Only scalar data representing dot products and coefficients in \texttt{axpy} operations as well as integer and floating point workspace to hold data for the tridiagonal subproblems is passed between the user and the library.
In particular, all vector data has to be managed by the user, who must be able to compute dot products $ \langle x, y \rangle $, perform \texttt{axpy} $ y := \alpha x + y $ on them and implement operator vector products $ x \mapsto Hx, x \mapsto M^{-1} x $ with the Hessian and the preconditioner.

Thus no assumption about representation and storage of vectorial data is made, as well as no assumption on the discretization of $ \mathcal H $ if $ \mathcal H $ is not finite-dimensional.
This is beneficial in problems arising from optimization problems stated in function space that may not be stored naturally as contiguous vectors in memory or where adaptivity regarding the discretization may be used along the solution of the trust region subproblem.
It also gives a trivial mechanism for exploiting parallelism in vector operations as vector data may be stored and operations may be performed on GPU without any changes in the trust region library.

In particular, this interface allows for easy interfacing with the PDE-constrained optimization software \texttt{DOLFIN-adjoint} \cite{Farrell2013,Funke2013} within the finite element framework \texttt{FEniCS} \cite{Alnaes2015,Logg2010,Alnaes2014} without having to rely on assumptions how the finite element discretization is stored, see~\S\ref{sec:fenics}.

\subsection{Conjugate Gradient Breakdown}

Per default, conjugate gradient iterations are used to build the Krylov subspace.
The algorithm switches to Lanczos iterations if the magnitude of the curvature $ \vert \langle \hat p, H \hat p \rangle \vert \le \emph{\texttt{tol\_curvature}} $ with a user defined tolerance $ \emph{\texttt{tol\_curvature}} \ge 0 $.

\subsection{Easy Case}

In the easy case after the Krylov space has been assembled in a particular iteration it remains to solve $ (\textup{TR}(T_i, \gamma^0 \boldsymbol{e_1}, I, \Delta, \mathbb R^{i+1})) $ which we do as outlined in \S\ref{sec:easy-tri}.
As mentioned there, an improved convergence order can be obtained by higher order Taylor expansions of $ \sigma_p(\lambda) $ and values $ p \neq -1 $, see \cite{Gould2010}.
However in our cases the computational cost for solving the tridiagonal subproblem --- often warmstarted in a suitable way --- is negligible in comparison the the cost of computing matrix vector products $ x \mapsto Hx $ and thus we decided to stick to the simpler Newton rootfinding on $ \sigma_{-1}(\lambda) $.

To obtain a suitable initial value $ \lambda^0 $ for the Newton iteration, we first try $ \lambda^* $ obtained in the previous Krylov iteration if available and otherwise $ \lambda^0 = 0 $.
If these fail, we use $ \lambda^0 = - \theta_{\min} $ computed as outlined in \S\ref{sec:newton-ini} by zero-finding on $ d(\theta) $ or $ \hat d(\theta) $.
This requires suitable models for $ \hat d(\theta) $.
Gould et al.~\cite{Gould1999} propose to use a quadratic model $ \theta^2 + a
\theta + b $ for $ \hat d(\theta) $ that captures the asymptotics $ t \to -
\infty $ obtained by fitting function value and derivative in a point in the
root finding process.
We have also had good success  with the linear Newton
model $ a \theta + b $, and with using a second order quadratic model $ a
\theta^2 + b \theta + c $, that makes use of an additional second derivative, as
well.
Derivatives of $ d(\theta) $ or $ \hat d(\theta) $ are easily obtained
by differentiating the recurrence for the Cholesky decomposition.
In our implementation a heuristic is used to select the
option that is inside the interval of safety and promises good progress.
The heuristic is given by using $ \theta^2 + a \theta + b $ in case that the bracket width $ \theta^k_{\text u} - \theta^k_\ell $ satisfies $ \theta^k_{\text u} - \theta^k_\ell \ge 0.1 \max\{1, \vert \theta^k \vert \} $ and $ a \theta^2 + b \theta + c $ otherwise.
The motivation behind this is that in the former case it is not guaranteed, that $ \theta^k $ has been determined to high accuracy \revise{as zero of $ d(\theta) $} and thus the model that captures the global behaviour might be better suited. \revise{I}n the latter case, $ \theta^k $ \revise{has been confirmed to be a zero of $ d(\theta) $} to a certain accuracy and it is safe to \revise{use} the model representing local behaviour.

\subsection{Hard Case}

\revise{We now discuss the so-called hard case of the trust region problem, which we have found to be of critical importance for the performance of trust region subproblem solvers in general nonlinear nonconvex programming. We discuss algorithmic and numerical choices made in \texttt{trlib} that we have found to help improve performance and stability.}

\subsubsection{Exact Hard Case}

The function for the solution of the tridiagonal subproblem implements the algorithm as given by Theorem~\ref{thm:hard-case-min} if provided with a decomposition in irreducible blocks.

However, from local information it is not possible to distinguish between convergence to a global solution of the original problem and the case in which an invariant Krylov subspace is exhausted that may not contain the global minimizer as in both cases the gradient vanishes.

The handling of the hard case is thus left to the user who has to decide in the reverse calling scheme if once arrived at a point where the gradient norm is sufficiently small the solution in the Krylov subspaces investigated so far or further Krylov subspaces should be investigated.
In that case it is left to the user to determine a new nonzero initial vector for the Lanczos iteration that is $ M $-orthogonal to the previous Krylov subspaces.
\revise{One possibility to obtain such a vector is using a random vector and $ M $-orthogonalizing it with respect to the previous Lanczos directions using the modified Gram-Schmidt algorithm.}

\subsubsection{Near Hard Case}

The near hard case arises if $ \langle \gamma^0 \boldsymbol{e_1}, \frac{\boldsymbol{\tilde v}}{\Vert \boldsymbol{ \tilde v} \Vert} \rangle $ is tiny, where $ \boldsymbol{\tilde v} $ spans the eigenspace $\text{Eig}(\theta_{\min}) = \text{span}\{\boldsymbol{\tilde v}\} $.

Numerically this is detected if there is no $ \lambda \ge \max\{0, -\theta_{\min}\} $ such that $ \Vert \boldsymbol x(\lambda) \Vert \ge \Delta $ holds in floating point airthmetic. In that case we use the heuristic $ \lambda^* = - \theta_{\min} $ and $ x^* = x(-\theta_{\min}) + \alpha \boldsymbol v $ with $ \boldsymbol v \in \text{Eig}(\theta_{\min}) $ where $ \alpha $ is determined such that $ \Vert \boldsymbol{x^*} \Vert = \Delta $.

Another possibility would be to modify the tridiagonal matrix $ T $ by dropping offdiagonal elements below a specified treshold and work on the obtained decomposition into irreducible blocks.
However we have not investigated this possibility as the heuristic seems to deliver satisfactory results in practice.

\subsection{Reentry with New Trust Region Radius}

In nonlinear programming applications it is common that after a rejected step another closely related trust region subproblem has to be solved with the only changed data being the trust region radius. As this has no influence on the Krylov subspace but only on the solution of the tridiagonal subproblem, efficient hotstarting has been implemented.
Here the tridiagonal subproblem is solved again with exchanged radius and termination tested.
If this point does not satisfy the termination criterion, conjugate gradient or Lanczos iterations are resumed until convergence.
However, we rarely observed the need to resume the Krylov iterations in practice.

\revise{An explanation is offered based on the use of the convergence criterion $$ \Vert \nabla L \Vert_{M^{-1}} \le \emph{\texttt{tol}} $$ as follows: In the Krylov subspace $ \mathcal K_i $, $$ \Vert \nabla L \Vert_{M^{-1}} = \gamma^{i+1} \vert \langle \bm x(\lambda), \bm{e_{i+1}} \rangle \vert \le \gamma^{i+1} \Vert \bm x(\lambda) \Vert_2 = \gamma^{i+1} \Delta.$$ Convergence occurs thus if either $ \gamma^{i+1} $  or the last component of $ \bm x(\lambda) \le \Delta $ are small. Reducing the trust region radius also reduces the upper bound for $ \Vert \nabla L \Vert_{M^{-1}} $, so convergence is likely to occur, especially if $ \gamma^{i+1} $ turns out to be small.}

\revise{If the trust region radius is small enough, or equivalently the Lagrange multiplier large enough, it can be proven that a decrease in the trust region radius leads to a decrease in $ \Vert \nabla L \Vert_{M^{-1}} $:
\begin{lemma}
    There is $ \hat \lambda \ge \max_i \vert \lambda_i(T) \vert $ such that $ \lambda \mapsto \gamma^{i+1} \vert \langle \bm{x}(\lambda), \bm{e_{i+1}} \rangle \vert $ is a decreasing function for $ \lambda \ge \hat \lambda $.
\end{lemma}
\begin{proof}
  Using the expansion $$ (T_i+\lambda I)^{-1} = \sum_{k \ge 0} (-1)^k \tfrac 1{\lambda^{k+1}} T^k, $$ which holds for $ \lambda \ge \max_{i} \vert \lambda_i(T) \vert $, we find:
\begin{align*}
    \Vert \nabla L \Vert_{M^-1} & = \gamma^{i+1} \vert \langle \bm x(\lambda), \bm{e_{i+1}} \rangle \vert = \gamma^{i+1} \gamma^0 \vert \langle (T_i+\lambda I)^{-1} \bm{e_1}, \bm{e_{i+1}} \rangle \vert \\
    & = \gamma^{i+1} \gamma^0 \left \vert \sum_{k\ge 0} (-1)^k \tfrac{1}{\lambda^{k+1}} \bm{e_{i+1}}^T T^k \bm{e_1} \right \vert = \frac{\prod_{j=0}^{i+1} \gamma^j}{\lambda^{i+1}} + O( (\tfrac 1\lambda)^{i+2}),
\end{align*}
where we have made use of the facts that $ \bm{e_{i+1}^T} T^k \bm{e_0} $ vanishes for $ k < i $, and that $ \bm{e_{i+1}^T} T^k \bm{e_0} = \prod_{j=1}^{i} \gamma^j $, which can be easily proved using the relation $ T \bm{e_j} = \gamma^{j-1} \bm{e_{j-1}} + \gamma^{j+1} \bm{e_{j+1}} + \delta_j \bm{e_j} $.
    The claim now holds if $ \lambda $ is large enough such that higher order terms in this expansion can be neglected.
\end{proof}
}

\subsection{Termination criterion}

Convergence is reported as soon as the Lagrangian gradient satisfies
\begin{align*}
    \Vert \nabla L \Vert_{M^{-1}} \le \begin{cases} \max\{\emph{\texttt{tol\_abs\_i}}, \emph{\texttt{tol\_rel\_i}} \, \Vert g \Vert_{M^{-1}} \}, & \text{if } \lambda = 0 \\
    \max\{\emph{\texttt{tol\_abs\_b}}, \emph{\texttt{tol\_rel\_b}} \, \Vert g \Vert_{M^{-1}} \}, & \text{if } \lambda > 0 \end{cases}.
\end{align*}
The rationale for using possibly different tolerances in the interior and boundary case is motivated from applications in nonlinear optimization where trust region subproblems are used as globalization mechanism.
There a local minimizer of the nonlinear problem will be an interior solution to the trust region subproblem and it is thus not necessary to solve the trust region subproblem in the boundary case to highest accuracy.

\subsection{\revise{Heuristic addressing ill-conditioning}}

\revise{The pL directions $ P_i $ are $ M $-orthogonal if computed using exact arithmetic.
  It is well known that, in finite precision and if $ H $ is ill-conditioned, $ M $-orthogonality may be lost due to propagation of roundoff errors .
  An indication that this happened may be had by verifying $$ \tfrac 12 \langle \bm h, T_i \bm h \rangle + \gamma^0 \langle \bm h, \bm{e_1} \rangle = q(P_i \bm h),$$
  which holds if $ P_i $ indeed is $ M $-orthogonal.
  On several badly scaled instances, for example \texttt{ARGLINB} of the \texttt{CUTEst} test set, we have seen that that both quantities above may even differ in sign,
  in which case the solution of the trust-region subproblem would yield a direction of ascent.
  This issue becomes especially severe if $ H $ has small, but positive eigenvalues and admits an interior solution of the trust region subproblem.
  Then, the Ritz values computed as eigenvalues of $ T_i $ may very well be negative due to the introduction of roundoff errors, and enforce a convergence to a boundary point of the trust region subproblem.
Finally, if the trust region radius $ \Delta $ is large, the two ``solutions'' can differ in a significantly.}

\revise{To address this observation, we have developed a heuristic that, by convexification, permits to obtain a descent direction of progress even if $ P_i $ has lost $ M $-orthogonality.
For this, let $ \underline{\rho} := \min_j \frac{\langle p_j, H p_j \rangle}{\langle p_j, Mp_j \rangle} $ and $ \overline \rho := \max_j \frac{\langle p_j, H p_j \rangle}{\langle p_j, Mp_j \rangle} $ be the minimal respective and Rayleigh quotients used as estimates of extremal eigenvalues of $ H $. Both are cheap to compute during the Krylov subspace iterations.
  \begin{enumerate}
    \item If algorithm~\ref{alg:KrylovTR} has converged with a boundary solution such that $ \lambda \ge 10^{-2} \max\{1, \rho_{\max} \} $ and $ \vert \rho_{\min} \vert \le 10^{-8} \rho_{\max} $, the case described above may be at hand. We compute $ q_x := q(P_i \bm{h}) $ in addition to $ q_h := \tfrac 12 \langle \bm h, T_i \bm h \rangle + \gamma^0 \langle \bm h, \bm{e_1} \rangle $.
      If either $ q_x > 0 $ or $ \vert q_x - q_h \vert > 10^{-7} \max\{ 1, \vert q_x \vert \} $, we resolve with a convexified problem.
    \item The convexification heuristic we use is obtained by adding a positive diagonal matrix $ D $ to $ T_i $, where $ D $ is chosen such that $ T_i + D $ is positive definite. We then resolve then the tridiagonal problem with $ T_i + D $ as the new convexified tridiagonal matrix.
      We obtain $ D $ by attempting to compute a Cholesky factor $ T_i $. Monitoring the pivots in the Cholesky factorization, we choose $ d_j $ such that the pivots $ \pi_j $ are at least slightly positive. The formal procedure is given in algorithm~\ref{alg:convexify}. Computational results use the constants $ \varepsilon = 10^{-12} $ and $ \sigma = 10 $.
  \end{enumerate}
}
   \begin{algorithm}
       \DontPrintSemicolon
       \SetKwInOut{Input}{\revise{input}}
       \SetKwInOut{Output}{\revise{output}}
       \SetKw{KwReturn}{\revise{return}}
       \Input{\revise{$ T_i $, $ \varepsilon > 0 $, $ \sigma > 0 $}}
       \Output{\revise{$ D $ such that $ T_i + D $ is positive definite}}
       \BlankLine  
       \For{\revise{$j = 0, \ldots, i$}}{
         \revise{{$ \hat \pi_j := \begin{cases} \delta_0, & j = 0 \\ \delta_j - \gamma_j^2 / \pi_{j-1}, & j > 0 \end{cases} $}}\;
         \revise{{$ d_j := \begin{cases} 0, & \hat \pi_j \ge \varepsilon \\ \sigma \vert \gamma_j^2/\pi_{j-1} - \delta_j \vert, & \hat \pi_j < \varepsilon \end{cases} $}}\;
         \revise{{$ \pi_j := \hat \pi_j + d_j $}}\;
       }
       \caption{\revise{Convexification heuristic for the tridiagonal matrix $ T_i $.}}
       \label{alg:convexify}
   \end{algorithm}

\subsection{TRACE}
\label{sec:trlib-trace}

In the recently proposed TRACE algorithm \cite{Curtis2016}, trust region problems are also used.
In addition to solving trust region problems, the following operations have to be performed:
\begin{itemize}
  \item $ \min_x \tfrac 12 \langle x, (H+\lambda M) x \rangle + \langle g, x \rangle $,
  \item Given constants $ \sigma_{\text l}, \sigma_{\text u} $ compute $ \lambda $ such that the solution point
    of $ \min_x \tfrac 12 \langle x, (H+\lambda M) x \rangle + \langle g, x \rangle $ satisfies $ \sigma_{\text l} \le \frac{\lambda}{\Vert x \Vert_{M}} \le \sigma_{\text u} $.
\end{itemize}

These operations have to be performed after a trust region problem has been solved and can be efficiently implemented using the Krylov subspaces already built up.

We have implemented these as suggested in \cite{Curtis2016}, where the first operation requires one backsolve with tridiagonal data and the second one is implemented as root finding on $ \lambda \mapsto \tfrac{\lambda}{\Vert x(\lambda) \Vert} - \sigma $ with a certain $ \sigma \in [\sigma_{\text l}, \sigma_{\text u}] $ that is terminated as soon as $ \tfrac{\lambda}{\Vert x(\lambda) \Vert} \in  [\sigma_{\text l}, \sigma_{\text u}] $.

\subsection{C11 Interface}
\label{sec:c_interfaces}

The algorithm has been implemented in C11.
The user is responsible for holding vector-data and invokes the algorithm by repeated calls to the function \texttt{trlib\_krylov\_min} with integer and floating point workspace and dot products $ \langle v, g \rangle, \langle p, Hp \rangle $ as arguments and in return receives status informations and instructions to be performed on the vectorial data.
A detailed reference is provided in the \texttt{Doxygen} documentation to the code.

\subsection{Python Interface}
\label{sec:py_interfaces}

A low-level python interface to the C library has been created using \texttt{Cython} that closely resembles the C API and allows for easy integration into more user-friendly, high-level interfaces.

As a particular example, a trust region solver for PDE-constrained optimization problems has been developed to be used from \texttt{DOLFIN-adjoint} \cite{Farrell2013,Funke2013} within \texttt{FEniCS} \cite{Alnaes2015,Logg2010,Alnaes2014}.
Here vectorial data is only considered as \texttt{FEniCS}-objects and no numerical data except of dot products is used of these objects.

    \section{Numerical Results}
\label{sec:results}

In this section, we present an assessment of the computational performance of
our implementation \texttt{trlib} of the GLTR method, and compare it to the
reference implementation \texttt{GLTR} as well as several competing methods for
solving the trust region problem and their respective implementations.

\subsection{Generation of Trust-Region Subproblems}
\label{sec:cutest}

For want of a reference benchmark set of non-convex trust region subproblems, we
resorted to the subset of unconstrained nonlinear programming problems of the
\texttt{CUTEst} benchmark library, and use a standard trust region
algorithm, e.g.~Gould et al.~\cite{Gould1999}, for solving $
\min_{\boldsymbol x \in \mathbb R^n} f(\boldsymbol x) $, as a generator of
trust-region subproblems. The algorithm starts from a given initial point $
\boldsymbol{x^0} \in \mathbb R^n $ and trust region radius $ \Delta^0 > 0 $, and
iterates for $ k \ge 0 $:

\begin{algorithm}[h]
    \DontPrintSemicolon
    \SetKwInOut{Input}{input}
    \SetKwInOut{Output}{output}
    \SetKw{KwReturn}{return}
    \Input{$f$, $x^0$, $\Delta^0$, $\rho_\textup{acc}$, $\rho_\textup{inc}$, $\gamma^+$, $\gamma^-$, \emph{\texttt{tol\_abs}}}
    \Output{$k$, $x^k$}
    \BlankLine  
    \For{$k\geq 0$}{
        {Evaluate $ \boldsymbol{g^k} := \nabla f(\boldsymbol{x^k}) $}\;
        {Test for termination: Stop if $ \Vert \boldsymbol{g^k} \Vert \le \emph{\texttt{tol\_abs}} $}\;
        {Evaluate $ H^k := \nabla^2_{\boldsymbol{xx}} f(\boldsymbol{x^k}) $}\;
        {Compute (approximate) minimizer $ \boldsymbol{d^k} $ to $ \text{TR}(H^k,\boldsymbol{g^k},I,\Delta^k) $}\;
        {Assess the performance  $ \rho^k := ({f(\boldsymbol{x^k}+\boldsymbol{d^k}) - f(\boldsymbol{x^k})})/{q(\boldsymbol {d^k})} $ of the step}\;
        {Update step: $\boldsymbol{x^{k+1}} := \begin{cases} \boldsymbol{x^k} + \boldsymbol{d^k}, & \rho^k \ge \rho_{\text{acc}} \\ \boldsymbol{x^k}, & \rho^k < \rho_{\text{acc}} \end{cases},$}\;
        {Update trust region radius: $\Delta^{k+1} := \begin{cases} \gamma^+ \Delta^k, & \rho^k \ge \rho_{\text{inc}} \\ \Delta^k, & \rho_{\text{acc}} \le \rho^k < \rho_{\text{inc}} \\ \gamma^- \Delta^k, & \rho^k < \rho_{\text{acc}} \end{cases}$}\;
    }
    \caption{Standard trust region algorithm for unconstrained nonlinear programming, used to generate trust region subproblems from \texttt{CUTEst}.}
    \label{alg:TRgenerator}
\end{algorithm}

In a first study, we compared our implementation \texttt{trlib} of the GLTR
method  to the reference implementation \texttt{GLTR} as well as several
competing methods for solving the trust region problem, and their respective
implementations, as follows:
\begin{itemize}
    \item \texttt{GLTR} \cite{Gould1999} in the \texttt{GALAHAD} library implements the GLTR method.
    \item \texttt{LSTRS} \cite{Rojas2008} uses an eigenvalue based approach.
        The implementation uses \texttt{MATLAB} and makes use of the direct \texttt{ARPACK} \cite{Lehoucq1998} reverse communication interface, which is deprecated in recent versions of \texttt{MATLAB} and lead to crashes within \texttt{MATLAB 2013b} used by us.
        We thus resorted to the standard \texttt{eigs} eigenvalue solver provided by \texttt{MATLAB} which might severly impact the behaviour of the algorithm.
    \item \texttt{SSM} \cite{Hager2001} implements a sequential subspace method that may use an SQP accelerated step.
    \item \texttt{ST} is an implementation of the truncated conjugate gradient method proposed independently by Steihaug \cite{Steihaug1983} and Toint \cite{Toint1981}.
    \item \texttt{trlib} is our implementation of the GLTR method.
\end{itemize}

All codes, with the exception of \texttt{LSTRS}, have been implemented in a
compiled language, Fortran 90 in case of \texttt{GLTR} and C in for all  other
codes, by their respective authors. \texttt{LSTRS} has been implemented in
interpreted MATLAB code. The benchmark code used to run this comparison has also
been made open source and is available as \texttt{trbench} \cite{Lenders2016b}.

In our test case the parameters $ \Delta^0 = \tfrac{1}{\sqrt{n}} $, $ \emph{\texttt{tol\_abs}} = 10^{-7} $, $ \rho_{\text{acc}} = 10^{-2} $, $ \rho_{\text{inc}} = 0.95 $, $ \gamma^+ = 2 $ and $ \gamma^- = \tfrac 12 $ have been used.
\revise{We used the subproblem convergence criteria as specified in table~\ref{tab:convcrit} for the different solvers, trying to have as comparable convergence criteria as possible within the available applications.
Our rationale for the interior convergence criterion to request $ \Vert \nabla L \Vert_{M^{-1}} = O( \Vert \bm{g^k} \Vert_{M^{-1}}^2 ) $ is that it defines an inexact Newton method with q-quadratic convergence rate, \cite[Thm 7.2]{Nocedal2006}.
As \texttt{LSTRS} is a method based on solving a generalized eigenvalue problem, its convergence criterion depends on the convergence criterion of the generalized eigensolver and is incomparable with the other termination criteria. With the exception of \texttt{trlib}, no other solver allows to specify different convergence criteria for interior and boundary convergence.}
\begin{table}
    \revise{\begin{tabular}{llll}
      solver & $ \tau $ interior convergence & $ \tau $ boundary convergence \\ \hline
      \texttt{GLTR} & $  \min\{ 0.5, \Vert \bm{g^k} \Vert_{M^{-1}} \} \Vert \bm{g^k} \Vert_{M^{-1}} $ & identical to interior \\
      \texttt{LSTRS} & \multicolumn{2}{l}{defined in dependence of convergence of implicit restarted Arnoldi method} \\
      \texttt{SSM} & $ \min\{ 0.5, \Vert \bm{g^k} \Vert \}_{M^{-1}} \Vert \bm{g^k} \Vert_{M^{-1}} $ & identical to interior \\
      \texttt{ST} & $ \min\{ 0.5, \Vert \bm{g^k} \Vert \}_{M^{-1}} \Vert \bm{g^k} \Vert_{M^{-1}} $ & method heuristic in that case \\
      \texttt{trlib} & $ \min\{ 0.5, \Vert \bm{g^k} \Vert \}_{M^{-1}} \Vert \bm{g^k} \Vert_{M^{-1}} $ & $ \max\{10^{-6}, \min\{ 0.5, \Vert \bm{g^k} \Vert_{M^{-1}}^{1/2} \} \} \Vert \bm{g^k} \Vert_{M^{-1}} $ 
    \end{tabular}}
    \caption{\revise{Convergence criteria for subproblem solvers $ \Vert \nabla L \Vert_{M^{-1}} \le \tau $}}
    \label{tab:convcrit}
\end{table}
\newcommand{\ca}{blue}
\newcommand{\cb}{blue!60}
\newcommand{\cc}{blue!25}
\newcommand{\cd}{blue!60}
\newcommand{\ce}{blue}
\begin{figure}[ht]
    \centering
    \includegraphics{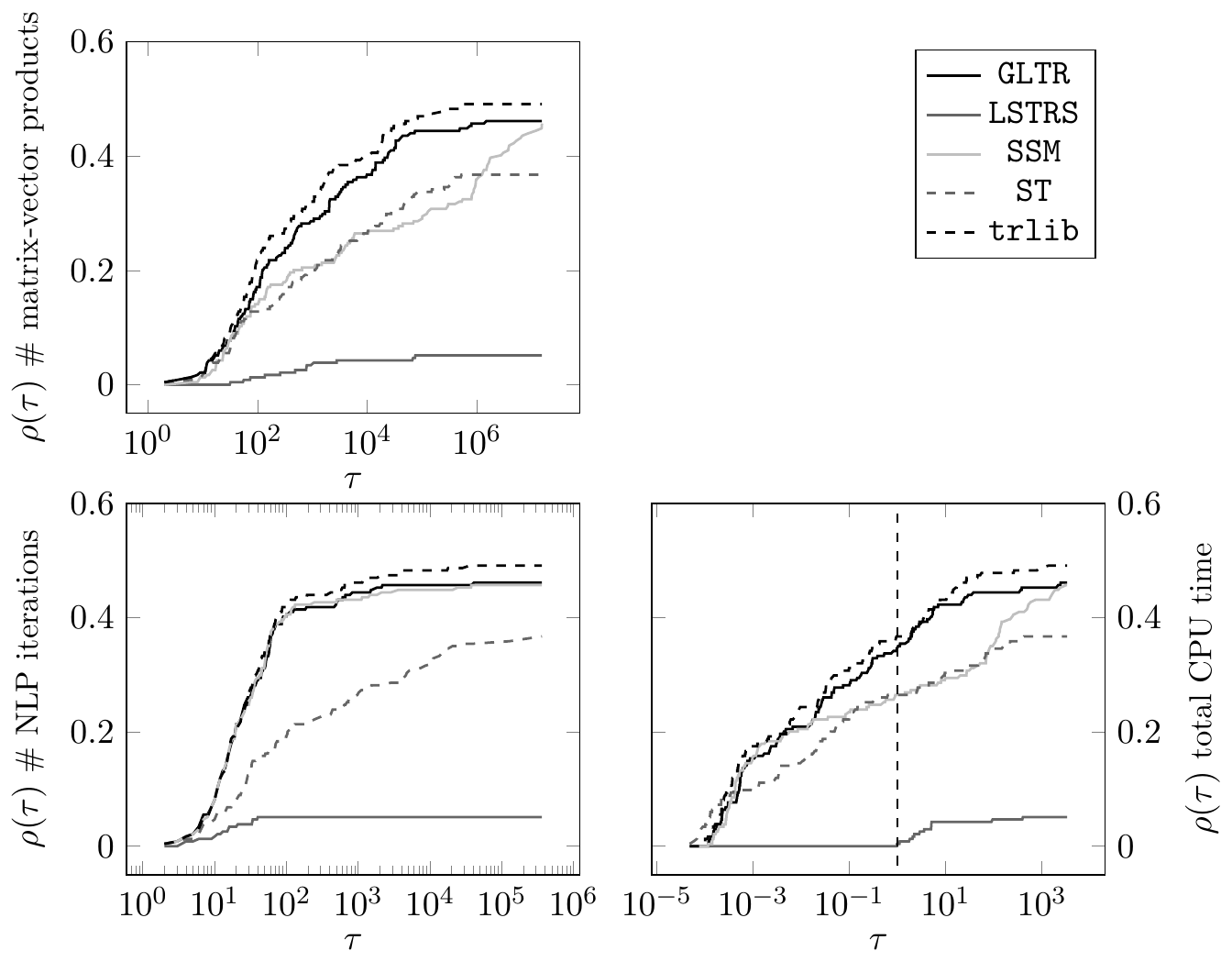}
    \caption{Performance Profiles for matrix-vector products, NLP iterations and total CPU time for different trust region subproblem solvers when used in a standard trust region algorithm for unconstrained minimization evaluated on the set of all unconstrained minimization problems from the \texttt{CUTEst} library.}
\end{figure}

The performance of the different algorithms is assessed using extended performance profiles as introduced by \cite{Dolan2002,Mahajan2011}, for a given set $ S $ of solvers and $ P $ of problems the performance profile for solver $ s \in S $ is defined by 
\begin{equation*}
  \rho_s(\tau) := \frac{1}{\vert P \vert} \vert \{ p \in P \, \vert \, r_{s,p} \le \tau \} \vert, \quad \text{where }r_{s,p} = \frac{t_{s,p} \text{ performance of } s \in S \text{ on } p \in P}{\min_{\sigma \in S, \sigma \neq s} t_{\sigma,p} }. 
\end{equation*} 

It can be seen that \texttt{GLTR} and \texttt{trlib} are the most robust solvers on the subset of unconstrained problems from \texttt{CUTEst} in the sense that they eventually solve the largest fraction of problems among all solversand that they are also among the fastest solvers.
That \texttt{GLTR} and \texttt{trlib} show similar performance is to be expected as they implement the identical GLTR algorithm, where \texttt{trlib} is slightly more robust and faster.
We attribute this to the implementation of efficient hotstart capabilities and also the Lanczos process to build up the Krylov subspaces once directions of zero curvature are encountered.
\revise{Tables~\ref{tab:cutestresults}--\ref{tab:cutestresults3} show the individual results on the \texttt{CUTEst} library.}

\begin{table}
\begin{tiny}
\begin{tabular}{l|l|ll|ll|ll|ll|ll}
  problem & $ n $ & \multicolumn{2}{c|}{\texttt{GLTR}} & \multicolumn{2}{c|}{\texttt{LSTRS}} & \multicolumn{2}{c|}{\texttt{SSM}} & \multicolumn{2}{c|}{\texttt{ST}} & \multicolumn{2}{c}{\texttt{trlib}} \\
          & & $ \Vert \nabla f \Vert $ & \# $ Hv $  & $ \Vert \nabla f \Vert $ & \# $ Hv $ & $ \Vert \nabla f \Vert $ & \# $ Hv $ & $ \Vert \nabla f \Vert $ & \# $ Hv $ & $ \Vert \nabla f \Vert $ & \# $ Hv $  \\
\texttt{AKIVA} & 2  & 3.7e-04 & 12 & 1.7e-03 & 104 & 3.7e-04 & 18 & 3.7e-04 & 12 & 3.7e-04 & 12 \\
\texttt{ALLINITU} & 4  & 1.2e-06 & 28 & 1.9e-05 & 275 & 1.2e-06 & 30 & 3.3e-05 & 20 & 1.2e-06 & 27 \\
\texttt{ARGLINA} & 200  & 2.1e-13 & 9 & 1.0e-13 & 485 & 2.8e-13 & 648 & 1.9e-13 & 10 & 1.8e-13 & 9 \\
\texttt{ARGLINB} & 200  & 1.4e-01 & 9 & 2.1e-01 & 14695 & \multicolumn{2}{c|}{failure} & 3.6e-04 & 152 & 9.7e-03 & 76 \\
\texttt{ARGLINC} & 200  & 7.9e-02 & 9 & 3.1e-01 & 9177 & \multicolumn{2}{c|}{failure} & 1.6e-03 & 156 & 5.1e-02 & 21 \\
\texttt{ARGTRIGLS} & 10  & 1.0e-09 & 50 & 3.6e-06 & 372 & 1.0e-09 & 15 & 1.2e-08 & 42 & 1.0e-09 & 50 \\
\texttt{ARWHEAD} & 5000  & 3.7e-11 & 20 & 2.4e-08 & 1054 & 3.7e-11 & 551752 & 3.7e-10 & 24 & 3.7e-11 & 17 \\
\texttt{BA-L16LS} & 66462  & 1.1e+06 & 58453 & 9.8e+07 & 83115 & \multicolumn{2}{c|}{failure} & 2.4e+06 & 20698 & 1.1e+08 & 21941 \\
\texttt{BA-L1LS} & 57  & 4.6e-08 & 317 & 1.3e+01 & 72289 & 6.0e-08 & 30336 & 1.2e-08 & 436 & 2.4e-08 & 758 \\
\texttt{BA-L21LS} & 34134  & 6.2e+06 & 129819 & 5.7e+07 & 208393 & 2.7e+09 & 1123576 & 1.2e+06 & 43139 & 9.8e+05 & 36639 \\
\texttt{BA-L49LS} & 23769  & 4.4e+04 & 250639 & 1.7e+06 & 1412516 & \multicolumn{2}{c|}{failure} & 2.9e+05 & 60741 & 8.7e+05 & 35305 \\
\texttt{BA-L52LS} & 192627  & 3.5e+08 & 21964 & 6.7e+09 & 36939 & \multicolumn{2}{c|}{failure} & 3.1e+07 & 16589 & 2.7e+07 & 19543 \\
\texttt{BA-L73LS} & 33753  & 1.4e+06 & 161282 & 7.1e+12 & 32865 & \multicolumn{2}{c|}{failure} & 7.5e+11 & 10071 & 4.7e+07 & 92020 \\
\texttt{BARD} & 3  & 5.6e-07 & 23 & \multicolumn{2}{c|}{failure} & 5.6e-07 & 24 & 9.8e-08 & 2910 & 5.6e-07 & 24 \\
\texttt{BDQRTIC} & 5000  & 5.7e-04 & 218 & 5.8e-04 & 4235 & 5.7e-04 & 811903 & 1.0e-02 & 529 & 5.7e-04 & 209 \\
\texttt{BEALE} & 2  & 1.2e-08 & 16 & 4.8e-06 & 93 & 1.2e-08 & 24 & 2.0e-08 & 62 & 1.2e-08 & 16 \\
\texttt{BENNETT5LS} & 3  & 6.5e-08 & 405 & \multicolumn{2}{c|}{failure} & 2.2e-04 & 2256 & 9.9e-08 & 876 & 1.8e-08 & 1691 \\
\texttt{BIGGS6} & 6  & 1.8e-08 & 71 & \multicolumn{2}{c|}{failure} & 5.9e-09 & 108 & 2.5e-04 & 20128 & 2.1e-08 & 410 \\
\texttt{BOX} & 10000  & 4.0e-04 & 32 & 6.8e-05 & 1021 & 4.0e-04 & 3278 & 1.8e-05 & 2172 & 4.0e-04 & 32 \\
\texttt{BOX3} & 3  & 6.6e-11 & 24 & \multicolumn{2}{c|}{failure} & 6.6e-11 & 24 & 1.0e-07 & 17266 & 6.6e-11 & 24 \\
\texttt{BOXBODLS} & 2  & 2.6e-01 & 50 & 7.8e-05 & 450 & 2.6e-01 & 87 & 3.8e-01 & 23 & 2.6e-01 & 42 \\
\texttt{BOXPOWER} & 20000  & 2.4e-08 & 86 & \multicolumn{2}{c|}{failure} & 1.6e+05 & 10285059 & 4.7e-05 & 1335136 & 5.6e-08 & 107 \\
\texttt{BRKMCC} & 2  & 6.1e-06 & 6 & 2.0e-08 & 74 & 6.1e-06 & 9 & 6.1e-06 & 6 & 6.1e-06 & 6 \\
\texttt{BROWNAL} & 200  & 2.8e-09 & 37 & \multicolumn{2}{c|}{failure} & 4.2e-10 & 128430 & 1.0e-07 & 54218 & 7.9e-10 & 32 \\
\texttt{BROWNBS} & 2  & 6.0e-06 & 75 & 1.1e-08 & 777 & 2.4e-07 & 99 & 8.9e-10 & 69 & 2.4e-07 & 67 \\
\texttt{BROWNDEN} & 4  & 7.3e-05 & 47 & 5.1e-04 & 268 & 7.3e-05 & 36 & 1.1e-01 & 54 & 7.3e-05 & 45 \\
\texttt{BROYDN3DLS} & 10  & 6.2e-11 & 60 & 4.5e-05 & 218 & 6.2e-11 & 18 & 1.4e-10 & 43 & 6.2e-11 & 57 \\
\texttt{BROYDN7D} & 5000  & 4.7e-04 & 13895 & 1.6e-04 & 201198 & 2.9e-04 & 2285169 & 1.2e-03 & 2206 & 5.8e-04 & 1660 \\
\texttt{BROYDNBDLS} & 10  & 2.0e-11 & 110 & 8.0e-05 & 466 & 2.0e-11 & 33 & 3.6e-13 & 70 & 2.0e-11 & 105 \\
\texttt{BRYBND} & 5000  & 6.2e-08 & 630 & 1.9e-06 & 93338 & 1.2e-09 & 3781397 & 7.6e-10 & 733 & 8.3e-13 & 639 \\
\texttt{CHAINWOO} & 4000  & 6.6e-04 & 40920 & 8.8e+02 & 69945 & 1.3e-04 & 3282530 & 1.5e-02 & 41073 & 3.1e-04 & 11482 \\
\texttt{CHNROSNB} & 50  & 5.2e-08 & 2032 & 8.1e-05 & 39963 & 3.5e-09 & 4008 & 1.8e-13 & 629 & 2.7e-10 & 1422 \\
\texttt{CHNRSNBM} & 50  & 1.5e-08 & 3181 & 1.0e-05 & 107423 & 4.4e-08 & 5065 & 1.4e-09 & 809 & 9.2e-09 & 1863 \\
\texttt{CHWIRUT1LS} & 3  & 5.4e+00 & 59 & 2.3e-01 & 139 & 2.1e-01 & 42 & 5.3e+00 & 43 & 2.1e-01 & 27 \\
\texttt{CHWIRUT2LS} & 3  & 4.0e-03 & 57 & 9.8e-02 & 138 & 3.4e-01 & 39 & 1.3e-02 & 37 & 3.4e-01 & 23 \\
\texttt{CLIFF} & 2  & 2.1e-05 & 38 & \multicolumn{2}{c|}{failure} & 2.1e-05 & 81 & 2.1e-05 & 41 & 2.1e-05 & 40 \\
\texttt{COSINE} & 10000  & 1.2e-06 & 213 & 7.2e+01 & 1 & 1.2e-06 & 6703 & 9.3e-03 & 72 & 1.2e-06 & 133 \\
\texttt{CRAGGLVY} & 5000  & 1.3e-04 & 622 & 1.2e-04 & 27113 & 1.3e-04 & 4646010 & 2.3e-03 & 453 & 1.3e-04 & 698 \\
\texttt{CUBE} & 2  & 1.2e-07 & 64 & 9.2e-06 & 564 & 2.6e-11 & 105 & 9.8e-08 & 204 & 2.6e-11 & 50 \\
\texttt{CURLY10} & 10000  & 3.7e-01 & 93106 & 1.3e+02 & 1 & 3.7e-01 & 1755070 & 1.8e-04 & 290643 & 4.5e-01 & 84837 \\
\texttt{CURLY20} & 10000  & 4.2e-03 & 94429 & 3.0e+02 & 1 & 2.5e-03 & 1334642 & 8.3e-02 & 98598 & 5.2e-03 & 96190 \\
\texttt{CURLY30} & 10000  & 2.7e-01 & 78302 & 4.2e+03 & 6974346 & 2.7e-01 & 146501 & 1.9e-02 & 128689 & 3.3e-01 & 77637 \\
\texttt{DANWOODLS} & 2  & 2.2e-06 & 18 & 5.6e-06 & 232 & 2.2e-06 & 27 & 2.2e-06 & 18 & 2.2e-06 & 18 \\
\texttt{DECONVU} & 63  & 2.4e-08 & 3650 & 1.3e-03 & 418777 & 4.0e-09 & 37199 & 2.3e-06 & 563021 & 8.3e-08 & 72328 \\
\texttt{DENSCHNA} & 2  & 6.6e-12 & 12 & 5.3e-08 & 136 & 6.6e-12 & 18 & 6.6e-12 & 12 & 6.6e-12 & 12 \\
\texttt{DENSCHNB} & 2  & 5.8e-10 & 12 & 1.3e-06 & 155 & 5.8e-10 & 18 & 1.0e-10 & 9 & 5.8e-10 & 12 \\
\texttt{DENSCHNC} & 2  & 8.7e-08 & 20 & 3.4e-06 & 237 & 8.7e-08 & 30 & 5.9e-08 & 20 & 8.7e-08 & 20 \\
\texttt{DENSCHND} & 3  & 5.1e-08 & 114 & \multicolumn{2}{c|}{failure} & 8.1e-08 & 135 & 3.7e-06 & 11399 & 8.1e-08 & 120 \\
\texttt{DENSCHNE} & 3  & 5.2e-12 & 35 & 9.5e-05 & 307 & 5.2e-12 & 45 & 2.1e-10 & 1442 & 5.2e-12 & 25 \\
\texttt{DENSCHNF} & 2  & 2.1e-09 & 12 & 3.6e-05 & 97 & 2.1e-09 & 18 & 1.0e-09 & 12 & 2.1e-09 & 12 \\
\texttt{DIXMAANA} & 3000  & 2.3e-13 & 44 & 1.5e-13 & 2763 & 2.3e-13 & 478120 & 6.7e-21 & 31 & 2.3e-13 & 38 \\
\texttt{DIXMAANB} & 3000  & 5.7e-08 & 503 & 7.3e-05 & 40355 & 5.7e-08 & 945986 & 1.6e-13 & 37 & 5.7e-08 & 80 \\
\texttt{DIXMAANC} & 3000  & 4.5e-12 & 1382 & 1.7e-05 & 40963 & 4.5e-12 & 1520049 & 4.5e-12 & 37 & 2.8e-09 & 95 \\
\texttt{DIXMAAND} & 3000  & 3.4e-13 & 1533 & 7.3e-08 & 68784 & 3.4e-13 & 1656761 & 2.7e-10 & 38 & 7.0e-17 & 169 \\
\texttt{DIXMAANE} & 3000  & 4.6e-08 & 2012 & \multicolumn{2}{c|}{failure} & 1.3e-11 & 3089 & 4.0e-11 & 515 & 1.6e-12 & 1281 \\
\texttt{DIXMAANF} & 3000  & 4.5e-08 & 2644 & \multicolumn{2}{c|}{failure} & 2.1e-08 & 1348070 & 1.0e-07 & 22275 & 6.7e-11 & 1079 \\
\texttt{DIXMAANG} & 3000  & 4.8e-08 & 4035 & 1.1e+00 & 845145 & 1.1e-08 & 1242789 & 1.0e-07 & 22211 & 2.0e-08 & 1673 \\
\texttt{DIXMAANH} & 3000  & 3.9e-08 & 5627 & 5.5e+02 & 1950740 & 5.9e-10 & 1696337 & 1.0e-07 & 22207 & 8.7e-08 & 2011 \\
\texttt{DIXMAANI} & 3000  & 1.0e-06 & 40507 & 1.0e+03 & 1 & 6.1e-06 & 19337 & 2.6e-07 & 3582057 & 1.8e-12 & 27353 \\
\texttt{DIXMAANJ} & 3000  & 4.6e-08 & 23746 & 2.2e+01 & 593623 & 6.2e-13 & 952725 & 1.8e-07 & 3314012 & 1.7e-07 & 11321 \\
\texttt{DIXMAANK} & 3000  & 4.6e-08 & 20831 & 1.5e+03 & 3100658 & 3.3e-11 & 1555718 & 1.8e-07 & 3310116 & 6.7e-07 & 14341 \\
\texttt{DIXMAANL} & 3000  & 4.6e-08 & 24371 & 3.1e+02 & 1122879 & 1.8e-09 & 1760641 & 1.8e-07 & 3319300 & 1.9e-11 & 16093 \\
\texttt{DIXMAANM} & 3000  & 4.7e-08 & 9845 & 4.4e+02 & 1 & 1.4e-11 & 2559 & 2.8e-07 & 4041601 & 1.0e-05 & 10745 \\
\texttt{DIXMAANN} & 3000  & 4.7e-08 & 33134 & 5.3e-01 & 1792578 & 4.5e-09 & 878377 & 1.9e-07 & 3874306 & 6.1e-08 & 18948 \\
\texttt{DIXMAANO} & 3000  & 4.8e-08 & 33105 & 1.1e-01 & 1810480 & 7.4e-08 & 968909 & 1.9e-07 & 3918576 & 3.4e-09 & 15832 \\
\texttt{DIXMAANP} & 3000  & 5.4e-08 & 19509 & 1.1e+02 & 90319 & 2.7e-08 & 1282847 & 2.8e-07 & 5486601 & 8.5e-10 & 12074 \\
\texttt{DIXON3DQ} & 10000  & 4.6e-08 & 40506 & 5.7e+00 & 1 & 6.1e-09 & 100140 & 1.3e-05 & 15308266 & 1.4e-12 & 19971 \\
\texttt{DJTL} & 2  & 3.9e+00 & 155 & 1.2e+05 & 1528 & 1.0e+01 & 3360 & 6.6e-01 & 1029 & 9.8e+00 & 2160 \\
\texttt{DMN15103LS} & 99  & 4.2e+01 & 924732 & 5.3e+03 & 177264 & 1.0e+02 & 87836914 & 7.8e+00 & 783230 & 6.6e+01 & 767826 \\
\texttt{DMN15332LS} & 66  & 2.7e-03 & 719233 & 8.1e+01 & 626859 & 3.6e+01 & 99777049 & 2.5e+00 & 1213511 & 2.5e+00 & 996706 \\
\texttt{DMN15333LS} & 99  & 1.5e+01 & 928176 & 2.7e+02 & 730749 & \multicolumn{2}{c|}{failure} & 5.4e+00 & 874786 & 2.9e+00 & 769091 \\
\texttt{DMN37142LS} & 66  & 9.4e-03 & 385536 & 3.1e+01 & 846259 & 1.4e-02 & 63711807 & 1.7e+00 & 1256055 & 1.7e+02 & 1073546 \\
\texttt{DMN37143LS} & 99  & 1.1e+00 & 547560 & 3.5e+03 & 84848 & 4.5e+00 & 41749169 & 1.4e+01 & 777991 & 1.3e+01 & 736780 \\
\texttt{DQDRTIC} & 5000  & 3.3e-10 & 39 & 8.3e-14 & 792 & 4.2e-12 & 3027385 & 1.3e-11 & 22 & 3.2e-10 & 25 \\
\texttt{DQRTIC} & 5000  & 4.1e-08 & 14236 & 1.3e+13 & 1 & 3.5e-08 & 15362086 & 1.0e-07 & 369300 & 3.5e-08 & 19244 \\
\texttt{ECKERLE4LS} & 3  & 1.8e-08 & 13 & \multicolumn{2}{c|}{failure} & 2.4e-08 & 63 & 1.6e-07 & 10001 & 2.4e-08 & 57 \\
\texttt{EDENSCH} & 2000  & 5.1e-05 & 342 & 9.5e-03 & 65271 & 5.1e-05 & 1645581 & 1.1e-04 & 147 & 5.1e-05 & 208 \\
\texttt{EG2} & 1000  & 2.9e-08 & 6 & \multicolumn{2}{c|}{failure} & 2.1e-04 & 1126 & 1.2e-02 & 11 & 2.9e-08 & 6 \\
\texttt{EIGENALS} & 2550  & 4.2e-07 & 9436 & \multicolumn{2}{c|}{failure} & 1.9e+00 & 276726 & 7.2e-08 & 151148 & 3.5e-09 & 5959 \\
\texttt{EIGENBLS} & 2550  & 6.5e-08 & 745535 & 4.8e+00 & 329779 & 6.8e-03 & 475261 & 3.3e-06 & 1132767 & 4.8e-05 & 1056840 \\
\end{tabular}
\end{tiny}
\caption{\revise{Results of subproblem solvers in individual \texttt{CUTEst} problems, part 1}}
\label{tab:cutestresults}
\end{table}

\begin{table}
\begin{tiny}
\begin{tabular}{l|l|ll|ll|ll|ll|ll}
  problem & $ n $ & \multicolumn{2}{c|}{\texttt{GLTR}} & \multicolumn{2}{c|}{\texttt{LSTRS}} & \multicolumn{2}{c|}{\texttt{SSM}} & \multicolumn{2}{c|}{\texttt{ST}} & \multicolumn{2}{c}{\texttt{trlib}} \\
          & & $ \Vert \nabla f \Vert $ & \# $ Hv $  & $ \Vert \nabla f \Vert $ & \# $ Hv $ & $ \Vert \nabla f \Vert $ & \# $ Hv $ & $ \Vert \nabla f \Vert $ & \# $ Hv $ & $ \Vert \nabla f \Vert $ & \# $ Hv $  \\
\texttt{EIGENCLS} & 2652  & 3.8e-08 & 796370 & \multicolumn{2}{c|}{failure} & 5.7e-01 & 402829 & 5.4e-09 & 66267 & 7.9e-09 & 270864 \\
\texttt{ENGVAL1} & 5000  & 2.4e-03 & 120 & 2.4e-03 & 18197 & 2.4e-03 & 3023116 & 5.9e-04 & 96 & 2.4e-03 & 107 \\
\texttt{ENGVAL2} & 3  & 6.5e-07 & 43 & 5.9e-06 & 353 & 4.5e-15 & 45 & 0.0e+00 & 42 & 1.7e-12 & 45 \\
\texttt{ENSOLS} & 9  & 9.3e-05 & 95 & 9.6e-05 & 412 & 9.3e-05 & 33 & 2.8e-04 & 68 & 9.3e-05 & 88 \\
\texttt{ERRINROS} & 50  & 7.3e-07 & 1446 & \multicolumn{2}{c|}{failure} & 9.0e-04 & 6582 & 7.6e-07 & 109821 & 9.2e-04 & 883 \\
\texttt{ERRINRSM} & 50  & 1.1e-03 & 2817 & \multicolumn{2}{c|}{failure} & 8.3e-03 & 5037 & 2.6e-06 & 720904 & 8.3e-03 & 1487 \\
\texttt{EXPFIT} & 2  & 2.1e-06 & 17 & 6.1e-07 & 131 & 4.8e-09 & 24 & 5.8e-06 & 17 & 4.8e-09 & 12 \\
\texttt{EXTROSNB} & 1000  & 9.9e-08 & 33028 & 2.3e-01 & 18905 & 5.7e-08 & 3716226 & 2.7e-06 & 12048850 & 1.0e-07 & 247139 \\
\texttt{FBRAIN2LS} & 4  & 2.8e-01 & 236 & \multicolumn{2}{c|}{failure} & 1.3e-02 & 138 & 4.5e-04 & 30008 & 1.3e-02 & 187 \\
\texttt{FBRAIN3LS} & 6  & 1.5e-06 & 60534 & \multicolumn{2}{c|}{failure} & 1.6e+01 & 486095 & 2.6e-03 & 39955 & 8.6e-08 & 30562 \\
\texttt{FBRAINLS} & 2  & 3.4e-05 & 14 & 3.9e-05 & 149 & 3.4e-05 & 21 & 8.6e-05 & 14 & 3.4e-05 & 14 \\
\texttt{FLETBV3M} & 5000  & 9.1e-03 & 4883 & \multicolumn{2}{c|}{failure} & 1.1e-03 & 19423 & 2.2e-05 & 885 & 2.6e-03 & 1379 \\
\texttt{FLETCBV2} & 5000  & \multicolumn{2}{c|}{failure} & \multicolumn{2}{c|}{failure} & \multicolumn{2}{c|}{failure} & \multicolumn{2}{c|}{failure} & \multicolumn{2}{c|}{failure} \\
\texttt{FLETCBV3} & 5000  & 3.1e+01 & 14194503 & 3.8e+01 & 55869908 & 3.2e+01 & 15365644 & 2.1e+01 & 4726900 & 3.0e+01 & 8099116 \\
\texttt{FLETCHBV} & 5000  & 2.7e+09 & 38547 & 3.7e+09 & 14764569 & 3.0e+09 & 35263513 & 3.6e+09 & 78 & 3.0e+09 & 18992 \\
\texttt{FLETCHCR} & 1000  & 4.2e-08 & 61120 & 7.0e-05 & 663337 & 4.8e-08 & 300564 & 4.2e-09 & 45367 & 4.8e-08 & 47342 \\
\texttt{FMINSRF2} & 5625  & 4.3e-08 & 12601 & 3.3e-01 & 1 & 6.4e-09 & 44273 & 5.1e-06 & 1931678 & 1.1e-09 & 3067 \\
\texttt{FMINSURF} & 5625  & 1.0e-07 & 8750 & 3.3e-01 & 1 & 5.8e-02 & 27451 & 6.8e-08 & 47015 & 8.7e-06 & 4011 \\
\texttt{FREUROTH} & 5000  & 3.9e-01 & 80 & 3.9e-01 & 4042 & 3.9e-01 & 6628218 & 6.0e-03 & 55 & 3.9e-01 & 69 \\
\texttt{GAUSS1LS} & 8  & 4.2e+01 & 68 & 1.1e+01 & 288 & 4.2e+01 & 21 & 1.4e+01 & 71 & 4.3e+01 & 60 \\
\texttt{GAUSS2LS} & 8  & 2.7e-01 & 79 & 2.3e-01 & 293 & 2.7e-01 & 24 & 1.4e+01 & 77 & 2.7e-01 & 70 \\
\texttt{GBRAINLS} & 2  & 1.4e-04 & 12 & 1.4e-04 & 94 & 1.4e-04 & 18 & 1.4e-04 & 12 & 1.4e-04 & 12 \\
\texttt{GENHUMPS} & 5000  & 4.8e-11 & 1486656 & 6.0e+03 & 1 & 4.7e-11 & 8692146 & 8.9e-08 & 35816 & 5.0e-12 & 529592 \\
\texttt{GENROSE} & 500  & 6.7e-04 & 16490 & 6.1e-05 & 309312 & 2.0e-06 & 66839 & 3.4e-05 & 3639 & 1.1e-04 & 8682 \\
\texttt{GROWTHLS} & 3  & 5.4e-03 & 345 & 3.2e-02 & 2027 & 8.9e-03 & 294 & 2.4e-03 & 4075 & 5.1e-05 & 239 \\
\texttt{GULF} & 3  & 4.0e-08 & 74 & \multicolumn{2}{c|}{failure} & 6.8e-08 & 78 & 5.7e-04 & 19576 & 6.8e-08 & 69 \\
\texttt{HAHN1LS} & 7  & 1.8e+03 & 9794 & 7.5e+01 & 5273 & 8.3e+01 & 332983 & 5.1e-01 & 5459 & 2.8e+00 & 592 \\
\texttt{HAIRY} & 2  & 1.7e-04 & 118 & 2.5e-05 & 993 & 1.2e-03 & 210 & 1.6e-03 & 137 & 1.2e-03 & 100 \\
\texttt{HATFLDD} & 3  & 2.1e-08 & 71 & \multicolumn{2}{c|}{failure} & 1.5e-11 & 75 & 1.0e-07 & 14033 & 1.5e-11 & 69 \\
\texttt{HATFLDE} & 3  & 3.5e-08 & 54 & \multicolumn{2}{c|}{failure} & 1.7e-10 & 57 & 9.8e-08 & 3318 & 1.7e-10 & 51 \\
\texttt{HATFLDFL} & 3  & 4.7e-08 & 283 & \multicolumn{2}{c|}{failure} & 6.6e-08 & 4404 & 5.1e-06 & 28015 & 3.5e-09 & 1078 \\
\texttt{HEART6LS} & 6  & 3.5e-08 & 6521 & 4.0e+00 & 29124 & 5.2e-08 & 3871 & 3.3e+00 & 39973 & 5.2e-08 & 8285 \\
\texttt{HEART8LS} & 8  & 4.0e-10 & 524 & 1.8e-05 & 1466 & 1.9e-09 & 147 & 2.0e-13 & 353 & 1.9e-09 & 379 \\
\texttt{HELIX} & 3  & 1.7e-11 & 36 & 3.4e-05 & 330 & 1.7e-11 & 36 & 3.7e-12 & 32 & 1.7e-11 & 36 \\
\texttt{HIELOW} & 3  & 5.4e-03 & 12 & 6.7e-03 & 87 & 5.4e-03 & 12 & 3.2e-05 & 18 & 5.4e-03 & 12 \\
\texttt{HILBERTA} & 2  & 2.8e-15 & 6 & 5.4e-15 & 56 & 2.2e-16 & 9 & 9.5e-08 & 301 & 6.2e-15 & 6 \\
\texttt{HILBERTB} & 10  & 2.4e-09 & 17 & 3.0e-06 & 202 & 2.4e-14 & 15 & 6.3e-10 & 12 & 2.4e-09 & 13 \\
\texttt{HIMMELBB} & 2  & 7.0e-07 & 18 & \multicolumn{2}{c|}{failure} & 2.1e-13 & 75 & 8.2e-13 & 33 & 1.2e-12 & 19 \\
\texttt{HIMMELBF} & 4  & 4.6e-05 & 308 & \multicolumn{2}{c|}{failure} & 4.6e-05 & 192 & 1.6e-02 & 29526 & 4.6e-05 & 287 \\
\texttt{HIMMELBG} & 2  & 8.6e-09 & 8 & 3.0e-05 & 62 & 8.6e-09 & 12 & 1.0e-13 & 11 & 8.6e-09 & 8 \\
\texttt{HIMMELBH} & 2  & 5.5e-06 & 8 & 7.7e-06 & 67 & 5.5e-06 & 15 & 5.0e-09 & 6 & 5.5e-06 & 9 \\
\texttt{HUMPS} & 2  & 1.0e-12 & 2955 & 4.7e-02 & 39232 & 3.1e-11 & 10767 & 1.0e-07 & 2297 & 2.6e-12 & 6202 \\
\texttt{HYDC20LS} & 99  & 1.1e-03 & 97095959 & 1.9e+06 & 738933 & \multicolumn{2}{c|}{failure} & 1.3e-01 & 93133732 & 1.3e-01 & 96002204 \\
\texttt{INDEF} & 5000  & 7.1e+01 & 297 & \multicolumn{2}{c|}{failure} & 7.1e+01 & 28565674 & 9.1e+01 & 6895561 & 7.1e+01 & 338 \\
\texttt{INDEFM} & 100000  & 1.1e-08 & 134 & \multicolumn{2}{c|}{failure} & \multicolumn{2}{c|}{failure} & 1.2e-02 & 3308 & 4.6e-09 & 92 \\
\texttt{INTEQNELS} & 12  & 2.3e-09 & 12 & 1.3e-05 & 145 & 4.9e-11 & 9 & 4.9e-11 & 15 & 4.9e-11 & 15 \\
\texttt{JENSMP} & 2  & 3.4e-02 & 18 & 3.4e-02 & 213 & 3.4e-02 & 27 & 3.4e-02 & 18 & 3.4e-02 & 18 \\
\texttt{JIMACK} & 3549  & 1.1e-04 & 103654 & 1.4e+00 & 1 & 9.4e-06 & 123549 & 9.1e-08 & 397707 & 8.8e-05 & 105680 \\
\texttt{KIRBY2LS} & 5  & 9.5e-03 & 198 & 5.1e+01 & 349 & 2.5e+00 & 60 & 4.2e+00 & 769 & 2.7e+00 & 83 \\
\texttt{KOWOSB} & 4  & 2.3e-07 & 40 & \multicolumn{2}{c|}{failure} & 1.0e-07 & 36 & 9.9e-08 & 8576 & 1.0e-07 & 40 \\
\texttt{KOWOSBNE} & 4  & 7.0e-08 & 124 & \multicolumn{2}{c|}{failure} & \multicolumn{2}{c|}{failure} & 1.0e-07 & 8375 & 2.4e-08 & 68 \\
\texttt{LANCZOS1LS} & 6  & 3.9e-08 & 484 & \multicolumn{2}{c|}{failure} & 5.2e-08 & 348 & 2.6e-05 & 29889 & 7.6e-08 & 651 \\
\texttt{LANCZOS2LS} & 6  & 3.7e-08 & 461 & 1.3e+02 & 1 & 1.5e-09 & 342 & 2.7e-05 & 29858 & 9.6e-08 & 625 \\
\texttt{LANCZOS3LS} & 6  & 4.1e-08 & 455 & \multicolumn{2}{c|}{failure} & 9.9e-08 & 393 & 2.6e-05 & 29950 & 2.6e-09 & 757 \\
\texttt{LIARWHD} & 5000  & 1.9e-08 & 44 & 3.9e-06 & 5072 & 1.9e-08 & 6202073 & 3.2e-14 & 168 & 1.9e-08 & 43 \\
\texttt{LOGHAIRY} & 2  & 9.2e-07 & 5102 & \multicolumn{2}{c|}{failure} & 8.1e-05 & 15966 & 1.5e-03 & 10003 & 1.5e-06 & 6676 \\
\texttt{LSC1LS} & 3  & 2.4e-07 & 74 & 1.2e-05 & 893 & 2.4e-07 & 81 & 5.7e-08 & 3057 & 2.4e-07 & 58 \\
\texttt{LSC2LS} & 3  & 2.2e-05 & 113 & \multicolumn{2}{c|}{failure} & 5.1e-05 & 156 & 3.8e-02 & 19975 & 9.1e-09 & 162 \\
\texttt{LUKSAN11LS} & 100  & 3.1e-12 & 14138 & 1.9e-07 & 103185 & 1.8e-12 & 800008 & 2.9e-13 & 2684 & 1.8e-12 & 9341 \\
\texttt{LUKSAN12LS} & 98  & 9.2e-03 & 675 & 3.7e-02 & 59360 & 9.2e-03 & 2545 & 1.5e-02 & 411 & 9.1e-03 & 402 \\
\texttt{LUKSAN13LS} & 98  & 5.5e-02 & 324 & 1.8e-02 & 6656 & 5.5e-02 & 18870 & 7.7e-04 & 176 & 5.7e-02 & 237 \\
\texttt{LUKSAN14LS} & 98  & 1.2e-03 & 580 & 1.3e-03 & 47362 & 1.2e-03 & 5703 & 4.2e-06 & 289 & 1.2e-03 & 349 \\
\texttt{LUKSAN15LS} & 100  & 4.7e-03 & 868 & 1.4e+00 & 559146 & 8.8e-04 & 4816 & 9.7e-08 & 1217 & 4.0e-04 & 758 \\
\texttt{LUKSAN16LS} & 100  & 1.2e-05 & 118 & 3.0e+04 & 1 & 1.2e-05 & 1229 & 9.2e-03 & 91 & 1.2e-05 & 123 \\
\texttt{LUKSAN17LS} & 100  & 4.9e-06 & 1043 & 1.5e-01 & 1653079 & 4.9e-06 & 6687 & 2.9e-05 & 1379 & 4.9e-06 & 1208 \\
\texttt{LUKSAN21LS} & 100  & 4.4e-08 & 2042 & 2.8e+00 & 1 & 7.7e-09 & 5922 & 3.3e-08 & 6962 & 7.3e-10 & 1750 \\
\texttt{LUKSAN22LS} & 100  & 7.5e-06 & 1122 & 1.5e-04 & 49915 & 3.6e-05 & 1456 & 1.8e-06 & 1251618 & 3.6e-05 & 893 \\
\texttt{MANCINO} & 100  & 3.4e-05 & 192 & 8.3e-05 & 5269 & 1.2e-07 & 206932 & 1.0e-07 & 45 & 1.1e-07 & 138 \\
\texttt{MARATOSB} & 2  & 9.8e-03 & 2639 & 8.7e+00 & 731 & 4.8e-02 & 3006 & 2.2e-02 & 1566 & 4.8e-02 & 1322 \\
\texttt{MEXHAT} & 2  & 2.0e-05 & 145 & 8.7e+01 & 753 & 6.6e-04 & 96 & 4.3e-04 & 60 & 6.6e-04 & 54 \\
\texttt{MEYER3} & 3  & 1.6e-03 & 1242 & 2.3e-03 & 7573 & 1.1e+03 & 933 & 4.1e-05 & 3780 & 8.9e-04 & 879 \\
\texttt{MGH09LS} & 4  & 1.7e-09 & 571 & \multicolumn{2}{c|}{failure} & 6.5e-10 & 369 & 6.5e-04 & 11810 & 2.1e-07 & 400 \\
\texttt{MGH10LS} & 3  & 7.2e+03 & 987 & 3.3e+06 & 140325 & 4.6e+05 & 552 & 7.4e+26 & 751 & 9.8e+03 & 193 \\
\texttt{MGH17LS} & 5  & 1.6e+00 & 41696 & \multicolumn{2}{c|}{failure} & 9.2e-06 & 4299 & 4.9e-06 & 39945 & 3.2e-05 & 772 \\
\texttt{MISRA1ALS} & 2  & 5.4e-04 & 89 & 2.4e-04 & 669 & 8.2e-02 & 297 & 1.3e-05 & 20002 & 3.5e-03 & 74 \\
\texttt{MISRA1BLS} & 2  & 1.1e-01 & 51 & 7.9e-02 & 481 & 3.0e-04 & 54 & 2.1e-04 & 20002 & 1.1e-01 & 50 \\
\texttt{MISRA1CLS} & 2  & 5.0e+00 & 44 & 3.2e-04 & 417 & 4.5e-04 & 48 & 2.4e-02 & 20002 & 5.0e+00 & 43 \\
\texttt{MISRA1DLS} & 2  & 1.3e+00 & 33 & 5.1e-03 & 271 & 3.1e-02 & 36 & 2.7e-03 & 20002 & 1.3e+00 & 32 \\
\texttt{MODBEALE} & 20000  & 4.3e-08 & 315 & 7.4e-05 & 419667 & 3.1e+05 & 10995895 & 6.6e-09 & 283 & 2.7e-11 & 385 \\
\texttt{MOREBV} & 5000  & 4.7e-08 & 4430 & 8.0e-04 & 1 & 7.4e-09 & 1126 & 1.6e-08 & 50000 & 1.6e-08 & 50001 \\
\end{tabular}
\end{tiny}
\caption{\revise{Results of subproblem solvers in individual \texttt{CUTEst} problems, part 2}}
\label{tab:cutestresults2}
\end{table}

\begin{table}
\begin{tiny}
\begin{tabular}{l|l|ll|ll|ll|ll|ll}
  problem & $ n $ & \multicolumn{2}{c|}{\texttt{GLTR}} & \multicolumn{2}{c|}{\texttt{LSTRS}} & \multicolumn{2}{c|}{\texttt{SSM}} & \multicolumn{2}{c|}{\texttt{ST}} & \multicolumn{2}{c}{\texttt{trlib}} \\
          & & $ \Vert \nabla f \Vert $ & \# $ Hv $  & $ \Vert \nabla f \Vert $ & \# $ Hv $ & $ \Vert \nabla f \Vert $ & \# $ Hv $ & $ \Vert \nabla f \Vert $ & \# $ Hv $ & $ \Vert \nabla f \Vert $ & \# $ Hv $  \\
\texttt{MSQRTALS} & 1024  & 4.6e-08 & 31351 & 1.0e+00 & 482955 & 8.7e-09 & 121693 & 6.4e-08 & 71336 & 7.6e-09 & 27636 \\
\texttt{MSQRTBLS} & 1024  & 4.7e-08 & 27153 & 9.4e-01 & 430163 & 6.4e-08 & 70400 & 6.9e-08 & 27457 & 1.1e-09 & 18431 \\
\texttt{NCB20} & 5010  & 1.2e-05 & 17788 & 2.8e+02 & 1 & 1.9e-08 & 144786 & 5.0e-06 & 45317 & 7.4e-04 & 5662 \\
\texttt{NCB20B} & 5000  & 4.3e-04 & 5964 & 2.8e+02 & 1 & 4.3e-04 & 42004 & 6.9e-04 & 4176 & 4.3e-04 & 3683 \\
\texttt{NELSONLS} & 3  & 4.5e+04 & 514 & 1.6e+05 & 55911 & 3.2e+04 & 1233 & 1.6e-03 & 560 & 2.4e+10 & 578 \\
\texttt{NONCVXU2} & 5000  & 7.3e-06 & 128020 & 3.2e+01 & 2607687 & 8.9e-06 & 8819837 & 1.3e-05 & 2616658 & 2.2e-04 & 41606 \\
\texttt{NONCVXUN} & 5000  & 1.4e-03 & 3407516 & 2.6e+01 & 2438251 & 1.5e-02 & 46939994 & 5.0e-03 & 3275980 & 2.3e-04 & 3292214 \\
\texttt{NONDIA} & 5000  & 4.6e-09 & 23 & 4.9e-07 & 1188 & 2.4e-09 & 5286524 & 6.1e-08 & 217 & 2.2e-09 & 19 \\
\texttt{NONDQUAR} & 5000  & 8.4e-08 & 44199 & 2.0e+04 & 1 & 1.9e-08 & 292931 & 4.1e-07 & 10001858 & 9.6e-08 & 148134 \\
\texttt{NONMSQRT} & 4900  & 8.3e+01 & 648705 & 2.7e+03 & 285111 & 3.3e+02 & 9434595 & 1.7e+00 & 604897 & 3.8e+02 & 590884 \\
\texttt{OSBORNEA} & 5  & 4.1e-08 & 220 & 1.1e-01 & 979 & 6.9e-06 & 126 & 2.6e-05 & 49955 & 6.9e-06 & 181 \\
\texttt{OSBORNEB} & 11  & 3.5e-07 & 409 & 9.4e-05 & 1570 & 7.2e-09 & 90 & 9.0e-08 & 4300 & 7.2e-09 & 314 \\
\texttt{OSCIGRAD} & 100000  & 6.2e-06 & 367 & 3.8e+05 & 1356648 & \multicolumn{2}{c|}{failure} & 6.6e-08 & 205 & 6.8e-08 & 380 \\
\texttt{OSCIPATH} & 10  & 2.5e-03 & 314 & 1.0e+00 & 1220 & 2.0e-02 & 7596 & 2.8e-04 & 80024 & 1.8e-02 & 65900 \\
\texttt{PALMER1C} & 8  & 3.8e-08 & 112 & \multicolumn{2}{c|}{failure} & 5.5e-08 & 1484 & 8.7e+00 & 78722 & 4.5e-08 & 91 \\
\texttt{PALMER1D} & 7  & 3.2e-08 & 70 & \multicolumn{2}{c|}{failure} & 2.3e-08 & 154 & 9.9e-07 & 34028 & 2.5e-08 & 63 \\
\texttt{PALMER2C} & 8  & 1.3e-08 & 83 & \multicolumn{2}{c|}{failure} & 1.5e-08 & 147 & 3.8e-03 & 69856 & 6.8e-09 & 71 \\
\texttt{PALMER3C} & 8  & 2.5e-09 & 84 & \multicolumn{2}{c|}{failure} & 5.8e-09 & 27 & 6.1e-03 & 69785 & 1.2e-09 & 73 \\
\texttt{PALMER4C} & 8  & 1.4e-08 & 96 & \multicolumn{2}{c|}{failure} & 8.8e-09 & 30 & 1.9e-02 & 69905 & 2.9e-09 & 89 \\
\texttt{PALMER5C} & 6  & 8.2e-14 & 39 & 6.3e-14 & 259 & 8.2e-14 & 24 & 8.3e-14 & 21 & 8.3e-14 & 31 \\
\texttt{PALMER6C} & 8  & 1.0e-08 & 92 & \multicolumn{2}{c|}{failure} & 4.6e-09 & 30 & 3.2e-01 & 58799 & 4.8e-09 & 79 \\
\texttt{PALMER7C} & 8  & 4.9e-08 & 121 & \multicolumn{2}{c|}{failure} & 4.5e-09 & 52 & 1.9e-02 & 59657 & 3.8e-08 & 109 \\
\texttt{PALMER8C} & 8  & 1.2e-09 & 111 & \multicolumn{2}{c|}{failure} & 3.5e-09 & 33 & 2.6e-01 & 58896 & 1.1e-09 & 97 \\
\texttt{PARKCH} & 15  & 4.5e-04 & 376 & 7.0e-02 & 1336 & 1.8e-04 & 63 & 6.6e-02 & 221 & 1.8e-04 & 287 \\
\texttt{PENALTY1} & 1000  & 2.3e-06 & 90 & \multicolumn{2}{c|}{failure} & 1.7e+13 & 2270123 & 1.0e-07 & 10284 & 2.9e-07 & 84 \\
\texttt{PENALTY2} & 200  & 1.2e+05 & 326 & 1.2e+05 & 8545 & 1.2e+05 & 61148 & 1.4e+02 & 169 & 1.2e+05 & 315 \\
\texttt{PENALTY3} & 200  & 2.5e-06 & 385 & \multicolumn{2}{c|}{failure} & 5.3e-08 & 101235 & 1.1e-07 & 1064 & 9.4e-08 & 762 \\
\texttt{POWELLBSLS} & 2  & 9.9e-08 & 162 & \multicolumn{2}{c|}{failure} & 6.3e-07 & 378 & 4.0e-04 & 20003 & 8.7e-08 & 139 \\
\texttt{POWELLSG} & 5000  & 9.9e-08 & 121 & \multicolumn{2}{c|}{failure} & 9.4e-08 & 816877 & 1.0e-07 & 381911 & 9.4e-08 & 136 \\
\texttt{POWER} & 10000  & 4.9e-08 & 12229 & 1.2e+14 & 30489 & \multicolumn{2}{c|}{failure} & 1.0e-07 & 13952 & 4.5e-08 & 16380 \\
\texttt{QUARTC} & 5000  & 4.1e-08 & 14236 & 1.3e+13 & 1 & 3.5e-08 & 15362086 & 1.0e-07 & 369300 & 3.5e-08 & 19244 \\
\texttt{RAT42LS} & 3  & 2.1e-01 & 81 & 1.3e-04 & 376 & 7.1e-05 & 66 & 1.6e-04 & 82 & 7.0e-05 & 56 \\
\texttt{RAT43LS} & 4  & 3.1e-01 & 143 & 1.1e+00 & 1098 & 3.1e-01 & 99 & 1.1e-01 & 428 & 3.1e-01 & 106 \\
\texttt{ROSENBR} & 2  & 9.3e-09 & 45 & 4.7e-06 & 521 & 3.9e-12 & 78 & 5.7e-11 & 46 & 3.9e-12 & 42 \\
\texttt{ROSZMAN1LS} & 4  & 9.3e-08 & 3380 & \multicolumn{2}{c|}{failure} & 1.1e-04 & 618 & 2.0e-04 & 29991 & 4.0e-06 & 114 \\
\texttt{S308} & 2  & 3.7e-06 & 18 & 6.4e-06 & 152 & 3.7e-06 & 27 & 1.8e-07 & 17 & 3.7e-06 & 18 \\
\texttt{SBRYBND} & 5000  & 1.3e+06 & 646854 & 2.6e+08 & 1 & 6.5e-08 & 15134337 & 9.7e+05 & 8066984 & 5.5e+03 & 411061 \\
\texttt{SCHMVETT} & 5000  & 2.2e-04 & 198 & 2.4e-04 & 5965 & 2.2e-04 & 2490 & 6.4e-03 & 175 & 2.2e-04 & 170 \\
\texttt{SCOSINE} & 5000  & 7.3e+02 & 9514524 & 3.3e+06 & 897980 & 9.7e-02 & 12118328 & 1.3e+05 & 22909922 & 7.9e+02 & 769553 \\
\texttt{SCURLY10} & 10000  & 1.7e+04 & 8358170 & 4.0e+07 & 1816763 & 3.7e+06 & 13209268 & 5.5e+05 & 10383078 & 9.0e+05 & 1175003 \\
\texttt{SCURLY20} & 10000  & 9.3e+04 & 5264236 & 7.9e+07 & 1762766 & 7.0e+06 & 9436240 & 2.7e+06 & 6338695 & 5.4e+05 & 1153609 \\
\texttt{SCURLY30} & 10000  & 2.5e+05 & 3928297 & 5.5e+07 & 1696073 & 1.0e+07 & 7548932 & 2.8e+06 & 4645780 & 1.2e+06 & 1087564 \\
\texttt{SENSORS} & 100  & 1.4e-04 & 351 & 1.4e-04 & 20908 & 1.4e-04 & 1207 & 1.6e-04 & 74 & 1.3e-04 & 226 \\
\texttt{SINEVAL} & 2  & 1.7e-07 & 101 & 2.0e-06 & 892 & 4.2e-17 & 174 & 5.4e-08 & 257 & 4.3e-17 & 81 \\
\texttt{SINQUAD} & 5000  & 5.4e+00 & 68 & 1.4e-01 & 6325 & 5.4e+00 & 481230 & 2.4e-02 & 38 & 5.4e+00 & 59 \\
\texttt{SISSER} & 2  & 4.3e-08 & 28 & 6.3e-05 & 229 & 4.3e-08 & 48 & 1.9e-07 & 10009 & 4.3e-08 & 32 \\
\texttt{SNAIL} & 2  & 5.0e-10 & 161 & 2.6e-05 & 1525 & 5.0e-10 & 297 & 2.6e-08 & 1232 & 5.0e-10 & 126 \\
\texttt{SPARSINE} & 5000  & 4.7e-08 & 490794 & 8.0e+02 & 724370 & 7.4e-09 & 11750907 & 3.3e-12 & 524818 & 1.4e-11 & 508898 \\
\texttt{SPARSQUR} & 10000  & 5.3e-08 & 937 & \multicolumn{2}{c|}{failure} & 4.5e-08 & 4741240 & 1.0e-07 & 20720 & 4.6e-08 & 1309 \\
\texttt{SPMSRTLS} & 4999  & 4.8e-08 & 2035 & 9.0e-05 & 160194 & 1.3e-08 & 5368 & 9.6e-12 & 4803 & 8.7e-14 & 1587 \\
\texttt{SROSENBR} & 5000  & 4.9e-12 & 28 & 9.6e-05 & 13105 & 4.9e-12 & 3503 & 9.2e-08 & 126 & 4.9e-12 & 28 \\
\texttt{SSBRYBND} & 5000  & 4.7e-08 & 74324 & 2.4e+06 & 244155 & 5.9e-09 & 5712 & 3.6e-08 & 252195 & 6.1e-10 & 83610 \\
\texttt{SSCOSINE} & 5000  & 3.9e+02 & 4072572 & 5.9e+03 & 1 & 1.5e+02 & 46643 & 1.7e-01 & 13991974 & 2.3e+02 & 11185306 \\
\texttt{SSI} & 3  & 4.7e-08 & 1692 & \multicolumn{2}{c|}{failure} & 8.8e-03 & 30003 & 2.2e-04 & 19968 & 3.1e-09 & 2919 \\
\texttt{STRATEC} & 10  & 4.2e-03 & 381 & 5.1e-01 & 1295 & 4.2e-03 & 78 & 3.3e-01 & 704 & 4.2e-03 & 291 \\
\texttt{TESTQUAD} & 5000  & 3.9e-08 & 2104 & 4.4e+07 & 1 & 1.8e-10 & 6723210 & 2.2e-13 & 3304 & 3.7e-10 & 2398 \\
\texttt{THURBERLS} & 7  & 4.2e-01 & 287 & 1.4e-01 & 1392 & 4.2e-01 & 2171 & 8.0e-03 & 1252 & 4.1e-01 & 203 \\
\texttt{TOINTGOR} & 50  & 2.7e-04 & 348 & 7.8e-05 & 30012 & 2.7e-04 & 990 & 6.0e-04 & 225 & 2.7e-04 & 351 \\
\texttt{TOINTGSS} & 5000  & 4.2e-08 & 148 & 3.0e-05 & 3827 & 4.2e-08 & 3893828 & 3.0e-05 & 147 & 3.2e-08 & 82 \\
\texttt{TOINTPSP} & 50  & 9.8e-06 & 450 & 8.0e-05 & 7546 & 4.5e-03 & 2842 & 7.5e-04 & 211 & 4.5e-03 & 248 \\
\texttt{TOINTQOR} & 50  & 4.0e-07 & 79 & 9.9e-05 & 2976 & 1.6e-09 & 458 & 4.9e-08 & 46 & 3.8e-07 & 88 \\
\texttt{TQUARTIC} & 5000  & 2.5e-07 & 32 & 6.4e-07 & 864 & 2.1e-14 & 2626 & 1.0e-07 & 83144 & 0.0e+00 & 35 \\
\texttt{TRIDIA} & 5000  & 4.7e-08 & 1064 & 9.8e-05 & 328271 & 2.5e-09 & 1070690 & 4.2e-14 & 1425 & 9.3e-12 & 1434 \\
\texttt{VARDIM} & 200  & 2.2e-09 & 33 & 8.5e-05 & 38751 & 9.5e-09 & 925519 & 1.7e-09 & 50 & 2.6e-09 & 33 \\
\texttt{VAREIGVL} & 50  & 3.8e-08 & 436 & 2.9e-07 & 13815 & 1.4e-10 & 2761 & 4.0e-08 & 425 & 7.7e-09 & 457 \\
\texttt{VESUVIALS} & 8  & 1.1e-02 & 821 & 4.2e+06 & 974 & 1.9e-02 & 11806 & 2.3e+02 & 69599 & 1.7e+01 & 802 \\
\texttt{VESUVIOLS} & 8  & 1.2e+01 & 382 & 1.5e+08 & 1 & 3.9e+02 & 3168 & 2.7e-01 & 11149 & 3.9e+02 & 187 \\
\texttt{VESUVIOULS} & 8  & 4.7e-03 & 157 & 2.4e+04 & 1417 & 9.7e-05 & 685 & 3.3e-02 & 131206 & 4.1e-03 & 237 \\
\texttt{VIBRBEAM} & 8  & 4.3e-04 & 465 & 4.5e+00 & 2609 & 2.0e-02 & 12818 & 3.7e-04 & 4213 & 1.2e-01 & 336 \\
\texttt{WATSON} & 12  & 4.7e-08 & 369 & \multicolumn{2}{c|}{failure} & 6.9e-09 & 57 & 1.8e-06 & 71328 & 9.4e-08 & 314 \\
\texttt{WOODS} & 4000  & 3.2e-12 & 250 & 6.1e+02 & 596298 & 1.4e-12 & 2771525 & 1.1e-10 & 317 & 1.4e-12 & 266 \\
\texttt{YATP1LS} & 2600  & 1.2e-09 & 59 & 8.1e-09 & 75319 & 8.6e-10 & 873596 & 1.0e-10 & 57 & 1.2e-09 & 52 \\
\texttt{YATP2LS} & 2600  & 5.7e-01 & 6160859 & 4.1e+02 & 5486629 & 1.1e+00 & 2139765 & 1.3e-10 & 35 & 4.4e-03 & 163257 \\
\texttt{YFITU} & 3  & 1.2e-05 & 166 & \multicolumn{2}{c|}{failure} & 4.7e-09 & 147 & 3.9e-03 & 29960 & 4.7e-09 & 137 \\
\texttt{ZANGWIL2} & 2  & 0.0e+00 & 2 & 1.9e-15 & 32 & 0.0e+00 & 6 & 0.0e+00 & 2 & 0.0e+00 & 2 \\
\end{tabular}
\end{tiny}
\caption{\revise{Results of subproblem solvers in individual \texttt{CUTEst} problems, part 3}}
\label{tab:cutestresults3}
\end{table}

\subsection{Function Space Problem}
\label{sec:fenics}

We solved a modified variant of \texttt{SCDIST1} \cite{Casas1986,Meyer2005} of the \texttt{OPTPDE} benchmark library \cite{OPTPDE,OPTPDE2014} for PDE constrained optimal control problems.
The state constraint has been dropped and a trust region constraint added in order to obtain the following function space trust region problem:
\begin{alignat*}{3}
    & \min_{y \in H^1(\Omega), u \in L^2(\Omega)} \quad   & \omit $ \displaystyle \tfrac 12 \Vert y - y_{\text{d}} \Vert_{L^2(\Omega)}^2 + \tfrac \beta 2 \Vert u - u_{\text{d}} \Vert_{L^2(\Omega)}^2 $ \span \span \\
  & \text{\,\,s.t.} &  -\triangle y + y & = u, & x \in \Omega \\
  &                 &  \partial_n  y   & = 0, & x \in \partial \Omega \\
  &                 &  \Vert y \Vert_{L^2(\Omega)}^2 + \Vert u \Vert_{L^2(\Omega)}^2 & \le \Delta^2
\end{alignat*}
Here $ \Omega \subseteq \mathbb R^n $, $ L^2(\Omega) $ denotes the Lebesgue space of square integrable functions $ f: \Omega \to \mathbb R $, $ H^1(\Omega) $ the sobolev space of square integrable functions that admit a square integrable weak derivative and $ \triangle: H $ is the Laplace operator $ \triangle = \sum_{i=1}^{n} \partial^2_{ii} $.

Tracking data $ y_{\text d}, u_{\text d} $ has been used as specified in \texttt{OPTPDE} where typical regularization parameters have been considered in the range $ 10^{-8} \le \beta \le 10^{-3} $.
Different geometries $ \Omega \in \{ (0,1)^2, (0,1)^3, \{ x \in \mathbb R^2 \, \vert \, \Vert x \Vert \le 1 \}, \{ x \in \mathbb R^3 \, \vert \, \Vert x \Vert \le 1 \} \} $ have been studied.

The finite element software \texttt{FEnICS} has been used to obtain a finite element discretization of the problem:
\begin{alignat*}{3}
    & \min_{\boldsymbol y \in \mathbb R^{n_y}, \boldsymbol u \in \mathbb R^{n_u}} \quad   & \omit $ \displaystyle \tfrac 12 \Vert \boldsymbol y - \boldsymbol y_{\text{d}} \Vert_{M}^2 + \tfrac \beta 2 \Vert \boldsymbol u - \boldsymbol u_{\text{d}} \Vert_{M}^2 $ \span \span \\
  & \text{\,\,s.t.} &  A \boldsymbol y - M \boldsymbol u& = 0, \\
  &                 &  \Vert \boldsymbol y \Vert_{M}^2 + \Vert \boldsymbol u \Vert_{M}^2 & \le \Delta^2,
\end{alignat*}
where $ M $ denotes the mass matrix and $ A = K + M $ with $ K $ being the stiffness matrix.

We used the approach suggested by Gould et al. \cite{Gould2001} to solve this equality constrained trust region problem:
\begin{enumerate} 
\item A null-space projection in the precondioning step of the Krylov subspace iteration is used to satisfy the discretized PDE constraint.
The required preconditioner is given by 
    $$ \begin{pmatrix} \boldsymbol y \\ \boldsymbol u \end{pmatrix} \mapsto \begin{pmatrix} I & 0 & 0 \\ 0 & I & 0 \end{pmatrix} \begin{pmatrix} M & 0 & A \\ 0 & M & -M \\ A & -M & 0  \end{pmatrix}^{-1} \begin{pmatrix} I & 0 \\ 0 & I \\ 0 & 0 \end{pmatrix} \begin{pmatrix} \boldsymbol y \\ \boldsymbol u \end{pmatrix}.$$
\item We used MINRES \cite{Paige1975} for solving with the linear system arising in this preconditioner to high accuracy. MINRES iterations themselves
are preconditioned using the approximate Schur-complement preconditioner 
$$ \begin{pmatrix} \tilde M & & \\ & \tilde M & \\ & & \tilde A M^{-1} \tilde A \end{pmatrix}^{-1}, $$ 
as proposed by \cite{Rees2010}. This preconditioner is an approximation to the optimal preconditioner 
$$ \begin{pmatrix} M & & \\ & M & \\ & & A M^{-1} A + M \end{pmatrix}^{-1} $$ 
that would lead to mesh-independent MINRES convergence in three iterations, provided exact arithmetic~\cite{Kuznetsov1995,Murphy1999} would be used.
\item In the MINRES preconditioner of step (2), products with $ \tilde M^{-1} $ and $ \tilde A^{-1} $ are computed using truncated conjugate gradients (CG) to high accuracy, again preconditioned using an algebraic multigrid as preconditioner.
\end{enumerate}

\begin{figure}
    \includegraphics{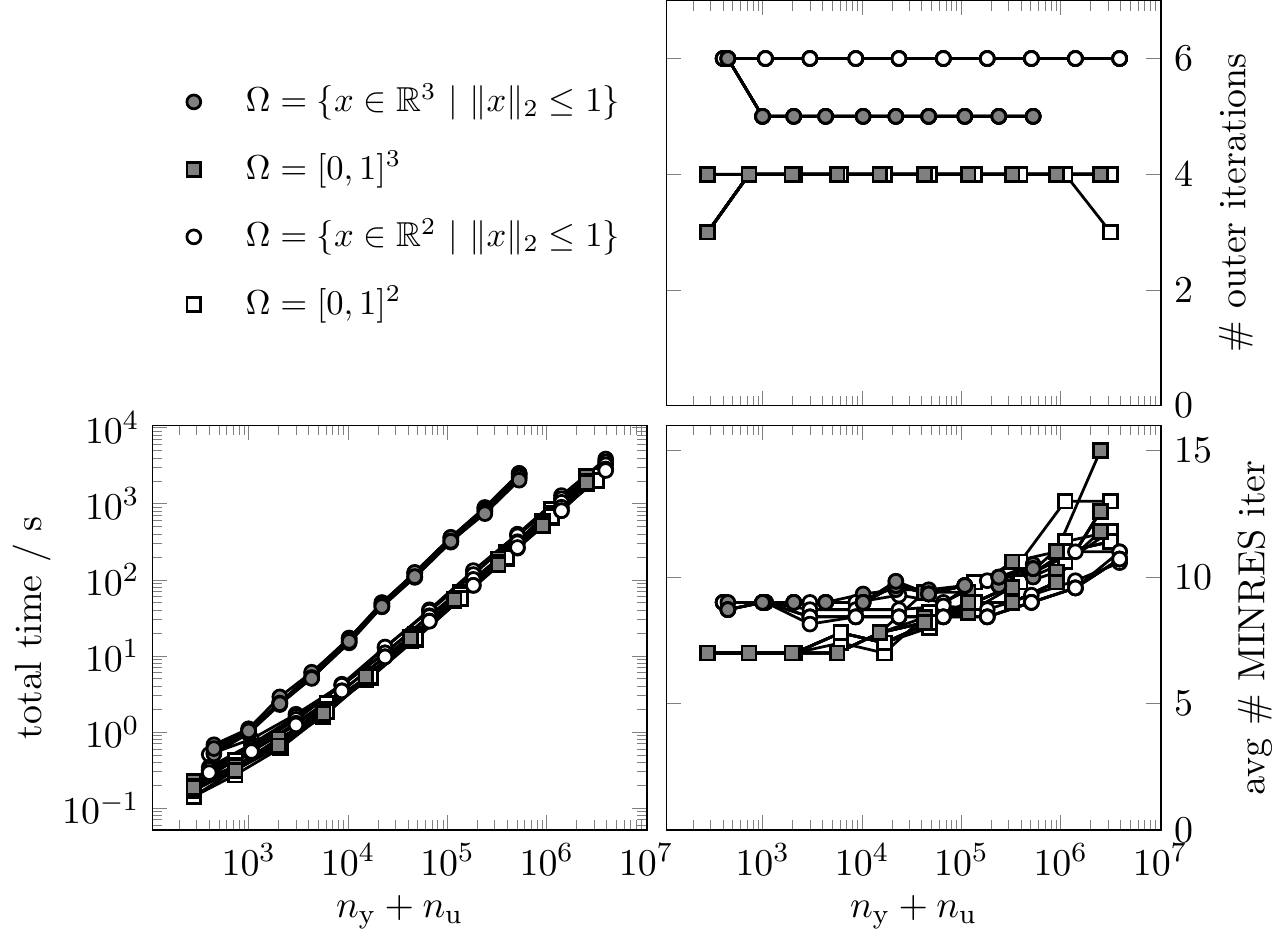}
    \caption{Results for distributed control trust region problem for different mesh sizes.
        Results are shown for four different geometries.
        Regularization parameters $ \beta \in \{10^{-3}, 10^{-4}, 10^{-5}, 10^{-6}, 10^{-7}, 10^{-8} \} $ have been considered, however computational results for a fixed geometry hardly change with $ \beta $ leading to near-identical plots.}
\label{fig:pde_all}        
\end{figure}

In Fig.~\ref{fig:pde_all}, it can be seen that using the GLTR method for these
function space problems yields a solver with mesh-independent
convergence behavior. The number of outer iterations is virtually constant on
a wide range of different meshes and varies at most by one iteration. The
number of inner (MINRES) iterations varies only slightly, as is to be
expected due to the use of an approximately optimal preconditioner in step (2).

    \section{Conclusion}
\label{sec:conclusions}

We presented \texttt{trlib} which implements Gould's
Generalized Lanczos Method for trust region problems.
Distinct features of the implementation are by the choice of a reverse communication interface that does not need access to vector data but only to dot products between vectors and by the implementation of preconditioned Lanczos iterations to build up the Krylov subspace.
The package \texttt{trbench}, which relies on \texttt{CUTEst}, has been introduced as a test bench for trust region problem solvers.
Our implementation \texttt{trlib} shows similar and favorable performance in comparison to the \texttt{GLTR} implementation of the Generalized Lanczos Method and also in comparison to other iterative methods for solving the trust region problem.

Moreover, we solved an example from PDE constrained optimization to show that the implementation can be used for problems stated in Hilbert space as a function space solver with almost discretization independent behaviour in that example.

    \section*{Funding}
    \small
    F.~Lenders acknowledges funding by the German National Academic Foundation.
    F.~Lenders and C.~Kirches acknowledge funding by DFG Graduate School 220,
    funded by the German Excellence Initiative. C.~Kirches acknowledges
    financial support by the European Union within the 7\textsuperscript{th}
    Framework Programme under Grant Agreement n\textsuperscript{o} 611909 and by
    Deutsche Forschungsgemeinschaft through Priority Programme 1962 ``Non-smooth
    and Complementarity-based Distributed Parameter Systems: Simulation and
    Hierarchical Optimization''.  C.~Kirches and A.~Potschka acknowledge funding
    by the German Federal Ministry of Education and Research program
    ``Mathematics for Innovations in Industry and Service'', grants
    n\textsuperscript{o} 05M2013-GOSSIP \revise{and 05M2016-MoPhaPro}.
    A. Potschka acknowledges funding by the European Research Council Adv. Inv. Grant MOBOCON 291 458.
    \revise{We are grateful to two anonymous referees who helped to significantly improve the exposition of this article and would like to thank R.~Herzog for bibliographic pointers regarding Krylov subspace methods in Hilbert space.}

\end{document}